\documentclass[reqno,10pt]{amsart}
\usepackage{amsmath,amssymb,bbm}
\usepackage{graphicx}

\newtheorem{theorem}{Theorem}[section]
\newtheorem{lemma}[theorem]{Lemma}

\newtheorem{example}[theorem]{Example}
\newtheorem{examples}[theorem]{Examples}

\numberwithin{equation}{section}
\def\vvvert{\hbox{\ensuremath{|\hspace{-0.16em}|\hspace{-0.16em}|}}}
\newcommand{\Norm}[3][{\vphantom 1}]%
  {\vvvert #2 \vvvert_{#3}^{#1}}

\newcommand {\Y}{{\mathcal Y}}

\newcommand {\kB}{{\mathcal B}}

\newcommand {\kC}{{\mathcal C}}
\newcommand {\kO}{{\mathcal O}}

\newcommand {\kL}{{\mathcal L}}

\newcommand {\kZ}{{\mathcal Z}}
\newcommand {\kH}{{\mathcal H}}
\newcommand {\kU}{{\mathcal U}}

\newcommand {\kY}{{\mathcal Y}}
\newcommand {\kX}{{\mathcal X}}

\newcommand {\1}{{\mathbbm{1}}}

\newcommand{\kCb}{\kC_{\operatorname{b}}}
\newcommand{\interior}{\operatorname{int}}

\newcommand {\D}{{\mathrm D}}
\renewcommand{\d}{{\mathrm d}}
\renewcommand{\i}{{\mathrm i}}
\newcommand {\e}{{\mathrm e}}
\newcommand {\spec}{\operatorname{spec}}
\renewcommand{\a}{{\mathsf a}}
\renewcommand{\b}{{\mathsf b}}
\renewcommand{\c}{{\mathsf c}}

\newcommand  {\R}{{\mathbb R}}
\newcommand  {\N}{{\mathbb N}}
\newcommand  {\C}{{\mathbb C}}

\newcommand  {\Z}{{\mathbb Z}}
\renewcommand{\P}{{\mathbb P}}

\newcommand  {\Q}{{\mathbb Q}}

\newcommand{\ip}[3][{\vphantom 1}]{\langle #2 \rangle_{#3}^{#1}}

\newcommand {\id}{\operatorname{id}}
\newcommand{\norm}[3][{\vphantom 1}]{\lVert #2 \rVert_{#3}^{#1}}

\begin{document}
\title[]
{Exponential Runge Kutta time semidiscetizations with low regularity initial data}

\author[C. Wulff]{Claudia Wulff}
\address[C. Wulff]%
{Department of Mathematics \\
 University of Surrey \\
 Guildford GU2 7XH \\
 UK}
\email{c.wulff@surrey.ac.uk}

\date{\today}

\begin{abstract}
  We  apply exponential Runge Kutta time discretizations
   to  semilinear evolution equations $ \frac {\d U}{\d t}=AU+B(U)$ posed
  on a Hilbert space $\kY$. Here $A$ is normal and generates a
  strongly continuous semigroup, and   $B$ is  assumed to be a 
smooth nonlinearity
  from $\kY_\ell = D(A^\ell)$ to itself, and  
$\ell \in I \subseteq [0,L]$, $L \geq 
  0$, $0,L \in I$.  In particular the  semilinear wave equation and nonlinear
  Schr\"odinger equation with periodic boundary
  conditions or posed on $\R^d$ fit into this framework. 
  We prove convergence  of order $O(h^{\min(\ell,p)})$
 for non-smooth initial data $U^0\in \kY_\ell$,  where $\ell >0$, for a method of classical order $p$. 
  We show in an example of an exponential Euler discretization of a linear evolution equation that our 
  estimates are sharp,  and corroborate this in numerical experiments for a semilinear wave equation. 
To prove our result we  Galerkin truncate the semiflow and numerical method and   balance the Galerkin truncation  error with the  error of the time discretization of the 
   projected system.
  We also extend these results to exponential Rosenbrock methods.

\end{abstract}

\maketitle

%\tableofcontents
%%%%%%%%%%%%%%%%%%%%%%%%%%%%%%%

\section{Introduction}
\label{sec:intro}
 
We analyze the convergence of exponential Runge Kutta  time
semidiscretizations of the semilinear evolution equation
\begin{equation}
  \frac{\d U}{\d t} = AU + B(U)
  \label{eqn:see}
\end{equation}
for low regularity initial data $U(0)=U^0$.   As in \cite{Wulff2016} we assume that (\ref{eqn:see}) is
posed on a Hilbert space $\kY$, $A$ is a normal linear operator  
that generates a strongly continuous semigroup, and that $B$ is smooth on a scale
 of Hilbert spaces $\{\kY_\ell\}_{\ell \in I}$, $I \subseteq [0,L]$, $0,L \in I$,
 see  condition (B)
below.  Here $\kY_\ell=D(A^\ell)\subseteq \kY$,
$\ell \geq 0$.
This condition is for example satisfied for the  semilinear wave equation and the nonlinear Schr\"odinger
equation in periodic domains or the full space with smooth nonlinearities, but, for $\ell>0$, poses additional restrictions in the case of other boundary conditions, see   \cite{Wulff2016}.

 Existence of the semiflow of (\ref{eqn:see}) is shown in \cite{Pazy}.  We discretize \eqref{eqn:see} in time by a possibly implicitly defined exponential Runge Kutta method of the class considered in  \cite{Owren} which, as we show,  is well-defined on $\kY$. At the end of the paper, in Section \ref{s.exp_Rosenbrock}, we also study exponential Rosenbrock methods as introduced in  \cite{SchweitzerHochbruck}.

 Given a time $T>0$  we prove an  order of convergence $\kO(h^\ell)$
   in the $\kY$ norm for the  time-semidiscretization up to time $T$  for any 
 solution  $U(t)$ of \eqref{eqn:see} with a given $\kY_\ell$ bound,  $\ell \in I$,
 for  $0<\ell \leq p$.  Here $\ell>0$ is such
that  $\ell-k \in I$ for $k=1,\ldots , \lfloor \ell \rfloor$ (the greatest integer $\leq \ell$),  and 
 $p$ is the order of the exponential integrator, i.e., the   order of the integrator if $A$ in \eqref{eqn:see}  is a bounded operator (e.g., if $\dim \kY<\infty$ so that  \eqref{eqn:see} is an ODE).
% For $\ell\geq p$ we prove  full order of convergence $\kO(h^p)$.
We show in an example of a linear evolution equation that this estimate is sharp (Example \ref{ex.EstSharp}) and  for a semilinear wave equation we demonstrate
numerical evidence as well, cf. Figure \ref{fig:expEuler_swe} below.
 
 We follow the  same strategy as in \cite{Wulff2016} where we proved an order of convergence $O(h^{\ell p/(p+1)})$ for A-stable Runge Kutta time semidiscretizations applied to semilinear evolution equations \eqref{eqn:see} for initial data in $Y_\ell$, $0<\ell<p+1$: our approach is to apply a spectral
Galerkin truncation to the  evolution equation
\eqref{eqn:see} and to   estimate   the  error of the time discretization of the projected evolution equation in terms
of the accuracy of the projection. We then balance  this error with the projection error to obtain an estimate
for the error of the time semi-discretization.

Related results in the literature are as follows:    in \cite{Hochbruck} full order
of convergence is shown for explicit exponential Runge Kutta methods (and other classes of explicit exponential integrators) in the case of  sufficiently smooth solutions  $t \to U(t)$ and/or nonlinearities $t \to B(U(t))$ and sectorial operators  $A$ under suitable order conditions.  
%  Exponential  Rosenbrock methods were   introduced in \cite{SchweitzerHochbruck}.
 In \cite{Ostermann_expBseries} order conditions for 
 smooth solutions of exponential Runge Kutta and
exponential Rosenbrock methods are derived. In \cite{Schweitzer} the author considers the exponential Euler Rosenbrock method
applied to a parabolic PDE for non-smooth initial data and proves (in general) fractional  order of convergence under certain smoothness assumptions of derivatives of the nonlinearity evaluated at the continuous solution. 
In \cite{Gauckler2015} the author studies trigonometric integrators applied to the
semilinear wave equation with  polynomial or analytic nonlinearity
on $S^1 = \R/\Z$ and proves error estimates for non-smooth initial data under
some conditions on the filter functions of the method. 
In \cite{OstermannBoussinesq2019} the authors carry out a coodinate transformation on the ``good'' Boussinesq equation that transforms it into a semilinear PDE with bounded linear part and design exponential integrators for the transformed PDE which, due to the special structure of the PDE, yields
higher convergence estimates than our results for non-smooth initial data in the original coordinates. 

 The order of convergence of splitting methods applied to semilinear evolution equations is studied in \cite{LubichNSE}, \cite{ConvAnalysisHigherOrdSplitSchroe2012},  \cite{Ostermann2014},c, see also references therein. In  \cite{LubichNSE} a  second order Strang splitting is applied to the nonlinear Schr\"odinger equation on $\R^3$  and convergence in the Sobolev  norm $\kH_2$  of order $1$ and in the $\kL_2$ norm of order $2$ is shown for initial data in the Sobolev space $\kH_4$. In \cite{ConvAnalysisHigherOrdSplitSchroe2012} the author studies convergence of  high order time splitting methods  with pseudospectral space discretizations of nonlinear Schr\"odinger equations where the nonlinearity has the form $B(U)U$
 and obtains full order in time convergence and high order spatial convergence for smooth initial data. In her analysis she uses   fractional order spaces $\kY_\ell$ as we do. 
 In  \cite{Ostermann2014} a Strang splitting is applied to the Vlasov-Poisson equation and full order of convergence is shown for smooth initial data.  In \cite{Gauckler2017} a Lie-Trotter time and Fourier space discretization is applied to the Zakharov system and convergence of order $1$ in time is proved under a CFL condition.

%%%%%%%%%%%%%%%%%%%%%%%%%%%%%%%%%%%%%%%%%%%%%%%%%%%%%%%%%%%%%%%%%%%%%%%%%%%%%%%%%%%%%%%%%%%%%%%%%%%

\section{Semilinear PDEs on a scale of Hilbert spaces}
\label{sec:hpde}
 
In this section we list our assumptions on  the semilinear evolution equation \eqref{eqn:see}.  These are the same as in \cite{Wulff2016}.
 %We review results from \cite{Pazy,OW10A} on the local well-posedness and regularity of solutions of  \eqref{eqn:see}  and give examples.

 We make the following assumptions
on the semilinear evolution equation \eqref{eqn:see}:

\begin{enumerate}
\item[(A1)]
 $A$ is a normal linear operator on $\kY$ that generates a
  strongly continuous semigroup of linear operators $e^{tA}$ on $\kY$
  in the sense of \cite{Pazy}.
\end{enumerate}
We define   $\kY_\ell = D(A^\ell )$, $\ell \geq
0$, $\kY_0=\kY$. 
For $m>0$ 
we define $\P_m$ to be the spectral projection of $A$ to $\spec(A) \cap \kB^m_\C(0)$,
 and let
$\Q_m=\id-\P_m$.
Here for a normed    space $\kX$ we let
\begin{equation*}
  \kB^R_{\kX}(U^0)=\{U\in\kX:\norm{U-U^0}{\kX}\leq R\}.
\end{equation*}
% %be the closed ball of radius $R$ around $U^0$ in $\kX$.
 We endow $\kY_\ell$ with the inner
product
\begin{equation}
  \ip{U_1,U_2 }{\kY_\ell}=\ip{\P_1 U_1,\P_1 U_2}{\kY} + \ip{|A|^\ell\Q_1 U_1,|A|^\ell\Q_1 U_2}{\kY},
  \label{eqn:inner_prod}
\end{equation}

In the following for $\ell \in \R$ let $\lfloor\ell\rfloor = \max\{n \in \N_0: n \leq \ell\}$ and  $\lceil \ell \rceil= \min\{n \in \N_0: n \geq \ell\}$.
Moreover for $R>0$ and $\ell\geq 0$ we abbreviate
$
 \kB_\ell^{R} = \kB_{\kY_\ell}^R(0).
$
We make the following assumption for the nonlinearity $B(U)$ of
\eqref{eqn:see}.

\begin{enumerate}
\item[(B)]
There exists  $L\geq 0$, $I \subseteq [0,L]$,
$0,L \in I$,
$N\in \N$, $N > \lceil L\rceil $,  such that $B\in
  \kCb^{N- \lceil \ell \rceil }(\kB^R_\ell;\kY_\ell)$ for
all  $\ell\in I$ and $R>0$.
  \label{enum:B}
\end{enumerate}
Here  for Banach spaces $\kX$, $\kZ$,  $\kU \subseteq \kX$, we denote by
 $\kCb^k(\kU,\kZ)$ the set of $k$ times
continuously differentiable functions $F \colon \interior\kU \to \kZ$ 
 such that $F$ and its derivatives $\D^i F$ are bounded as 
maps from  the interior $\interior\kU$ of $\kU$ to
the  space of $i$-multilinear bounded maps from $\kX$ to $\kZ$
 and extend continuously to the boundary of $\interior\kU$ for $i\leq k$.
We set $\kCb(\kU,\kZ)= \kCb^0(\kU,\kZ)$.
%Note that if $\dim \kX=\infty$, there are examples of continuous functions $F:\kU\to \kZ$ where $\kU$ is closed and bounded, which do not lie in $\kCb(\kU,\kZ)$, see e.g. \cite[Remark 2.3]{OW10A}.

We denote the supremum  of
$B:\kB_\ell^R\to\kY_\ell$ as $M_\ell[R]$ and the supremum of its   $k$th derivative as
$M_\ell^{(k)}[R]$ and set $M'_\ell[R]= M^{(1)}_\ell[R]$, $M[R] =M_0[R] $ and $M'[R] =M'_0[R]$.  
Moreover we define
\begin{equation}\label{e.I}
I^-:= \{ \ell \in I, \ell -k\in I, k =1,\ldots, \lfloor\ell \rfloor\}.
\end{equation}
We write
$\Phi^t(U^0)\equiv\Phi(U^0,t)\equiv U(t)$ for the solution of  (\ref{eqn:see})
 with initial value   $U(0)=U^0\in\kY$ which exists on some time interval $[0,T]$, $T>0$ by \cite{Pazy}. 
The following  theorem \cite[Theorem 2.2]{Wulff2016}
provides additional regularity of the semiflow $\Phi^t$ under our assumptions.

\begin{theorem}[Regularity of the semiflow]\label{t.semiflow} 
  Assume   (A1) and (B). Let $R>0$. Then
  there is $T_*>0$ such that there
  exists a semiflow $\Phi$  of  \eqref{eqn:see}
which satisfies
\begin{subequations}
  \begin{equation}
    \Phi^t\in \kCb^{N}(\kB_0^{R/2};\kB^R_0) 
   \label{e.PhiUDeriv}
  \end{equation}
with uniform bounds in  $t\in [0,T_*]$. Moreover if $\ell \in I^-$, 
 $k\in \N_0$, 
$k \leq  \ell$,  then
\begin{equation}
 \Phi(U) \in \kCb^k([0,T_*];\kB^R_0) 
 \label{eqn:semiflow_regt}
\end{equation}
with uniform bounds in $U \in \kB_\ell^{R/2}$.
\label{e.semiflow_reg}
\end{subequations}
  The bounds on $T_*$ and   $\Phi$    depend only on $R$,  
  $\omega$ from \eqref{eqn:semigroup_bound}, and the bounds  afforded by
  assumption (B) on balls of radius $R$.
\end{theorem}

Examples of PDEs satisfying assumptions (A) and (B) are the semilinear wave equation and nonlinear Schr\"odinger equation with periodic, Neumann and Dirichlet boundary
  conditions as discussed in \cite{Wulff2016}.
%\fbox{include shortened examples}

Assumption (A1) implies that 
\begin{equation}
\Re(\spec(A))\leq \omega, \quad  \norm{e^{tA}}{\kY\to\kY}\leq   e^{\omega t},
  \label{eqn:semigroup_bound}
\end{equation}
 for some $\omega \geq 0$, see \cite{Pazy}. 
Let us decompose $A$ as $A_{\rm skew} = \frac12(A-A^*)$, $A_{\rm sym}= \frac12(A+A^*)$.
Then $A_{\rm skew}$ and $A_{\rm sym}$ commute, $A=A_{\rm skew}+  A_{\rm sym}$, $A_{\rm skew}$ is a skew symmetric operator and
$A_{\rm sym}$ is self-adjoint  and its spectrum is bounded from above by \eqref{eqn:semigroup_bound}:    $ {\rm spec}(A_{\rm sym}) \leq \omega$.  Let $\P^\pm_{\rm sym}$ the spectral projection
of $A_{\rm sym}$ to $\R^\pm$  
and $A^\pm_{\rm sym}:= \P^\pm_{\rm sym} A_{\rm sym}$.
From now on we assume without loss of generality that $\omega=0$ in \eqref{eqn:semigroup_bound} 
by adding $A_{\rm sym}^+ U$ to $B(U)$ and replacing $A$ by $A_{\rm skew} + A^-_{\rm sym}$.
 
\begin{enumerate}
\item[(A2)] $A$ satisfies \eqref{eqn:semigroup_bound} with $\omega=0$.
\end{enumerate}

%%%%%%%%%%%%%%%%%%%%%%%%%%

\section{Exponential Runge Kutta methods}
As in \cite{Owren} we consider numerical methods of the form
\begin{subequations}
\label{e.expInt}
\begin{align}
W &= \exp(h \c A) U^0\1 + h  \a (h A) B(W ),
\label{e.Pi}\\
U^1 &= \exp(h  A) U^0  + h  \b^T(hA) B(W). \label{e.psi}
\end{align}
\end{subequations}
which are called exponential Runge Kutta methods.
Here we define
\begin{equation}
  U \1 = 
  \begin{pmatrix}
    U \\
    \vdots \\
    U
  \end{pmatrix}\in\kY^s ~~\mbox{ for }~~U \in \kY,\quad
  W =
  \begin{pmatrix}
    W^1 \\
    \vdots \\
    W^s
  \end{pmatrix},\quad B(W) =
  \begin{pmatrix}
    B(W^1) \\
    \vdots \\
    B(W^s)
  \end{pmatrix}.
  \label{e.defB(W)}
\end{equation}
$W^1,\ldots,W^s$ are the stage vectors, $0\leq \c_1\leq \c_2\leq \ldots \leq \c_s\leq 1$  and we denote
\begin{equation*}
  (\a W)^i=\sum_{j=1}^s\a_{ij}W^j,\quad \b^TW = \sum_{i=1}^s \b_iW^i, \quad
    \c W = \left( \begin{array}{c}   \c_1  W^1\\
 \c_2 W^2\\
\vdots\\
\c_s  W^s
\end{array} \right).
\end{equation*}
We define
\begin{equation}\label{e.normW}
\|W\|_{\kY_\ell^s} := \max_{j=1,\ldots,s} \|W^i\|_{\kY_\ell}.
\end{equation}

\begin{examples}\rm \label{ex.expRK} \
\begin{enumerate}
\item For  the {\em exponential Euler method}
\[
U^1  = \exp(h  A) U^0  + h \varphi_1(h  A)  B(U^0)
\]
 we have $s=1$, $\b(z) = \varphi_1(z)$,  $\a=\c=0$ and $\varphi_1(z) =\frac{e^{z}-1}{z} $, see e.g.~\cite{Hochbruck}. 
 %This method solves systems with $B$ independent of $U$ exactly.
\item   For  the {\em Euler-Larson method} 
$
U^1  = \exp(h  A) U^0  + h \exp(h  A)  B(U^0)
$
we have  $s=1$, $\b(z) = \exp(z)$,  $\a=\c=0$.
\item For the 
{\em implicit Lawson-Euler method}  
$
U^1  = \exp(h  A) U^0  + h    B(U^1)
$
we have $s=1$, $\b(z) = 1$,  $\a=\c=1$.
\end{enumerate}
\end{examples}

 We assume the following:
 \begin{enumerate}
 \item[(EXP)]   $\a: \C^-_0 \to {\rm Mat}_\C(s,s)$, $\b: \C^-_0 \to \C^s$ and their derivatives are analytic and bounded with bounds   $M_\a^{(k)} =\sup_{z \in  \C^-_0} \|\a^{(k)}(z)\|_\infty$, $k\in \N_0$, $ M_\a =M_\a^{(0)} $, 
$M_\a'=M_\a^{(1)}$,  and $M_\b^{(k)}= \sup_{z \in  \C^-_0} \|\b^{(k)}(z)\|_1$, $M_\b^{(0)}= M_\b$, $M_\b'= M_\b^{(1)}$.
 \end{enumerate}
 
 Assumption (EXP)  is true for the  examples above. Note that $\a(hA)$ and $\b(hA)$ are well defined by functional calculus because $A$ is a normal operator.

%%%%%%%%%%%%%%%%%%%%%%%%%%%%%%%%%%%%%%

\section{Regularity of the exponential Runge-Kutta method}
To prove regularity of the exponential Runge-Kutta method we use the following estimate which follow from \eqref{eqn:inner_prod}: 
\begin{equation}
 \label{e.Y_ellCompare}
\|A^\ell\|_{\kY_\ell \to \kY} \leq 1~~\mbox{for}~~\ell\geq 0 \quad\mbox{and}\quad \| U\|_{\kY_\ell} \leq \| U\|_{\kY_n} ~~\mbox{for}~~ 0\leq \ell\leq n
\end{equation}

We also need the following lemma:

\begin{lemma}[Bounds on the $h$ derivatives of $\a(hA)$ and $\b(hA)$]\label{l.ahA}
 Assume (EXP), (A1) and (A2).  Then for all $\ell\geq 0$, $k\in \N_0$, $\ell\geq k$, $h\geq 0$, $W\in \kY^s_\ell$, 
\[ h \mapsto \a(hA) W \in \kCb^k(\R_0^+; \kY^s_{\ell-k}),\quad  h \mapsto \b(hA) W \in \kCb^k(\R_0^+; \kY_{\ell-k})
\]
and 
\begin{subequations}
\begin{align}
\|\partial_h^k ( \a(hA)) \|_{\kY^s_\ell\to \kY^s_{\ell-k}}\leq M_\a^{(k)}, \quad |\partial_h^k ( \b(hA)) \|_{\kY^s_\ell\to \kY_{\ell-k}}\leq M_\b^{(k)},
\\\
   \|\partial^k_h( h\a(hA) ) W\|_{\kY^s_{\ell-k}}  \leq
k M_\a^{(k-1)} \|W\|_{\kY^s_{\ell-1}}+   h M_\a^{(k)} \|W\|_{\kY^s_\ell} \leq M_{h,\a}^{(k)} \|W\|_{\kY^s_\ell}, \label{e.hahA}\\
 \|\partial^k_h( h\b(hA) ) W\|_{\kY_{\ell-k}}  \leq
k M_\b^{(k-1)} \|W\|_{\kY^s_{\ell-1}}+   h M_\b^{(k)} \|W\|_{\kY^s_\ell} \leq M_{h,\b}^{(k)} \|W\|_{\kY^s_\ell},\label{e.hbhA}
\end{align}
\end{subequations}
with  $M_\a^{(k)} $, $M_\b^{(k)} $ from (EXP) and $M_{h,\a}^{(k)} = k M_\a^{(k-1)}+ h M_\a^{(k)}$,   $M_{h,\b}^{(k)} = k M_\b^{(k-1)}+ h M_\b^{(k)}$.
\end{lemma}
\begin{proof}
We define $\P_m W$ such that $(\P_m W)^i= \P_m W^i$, $i=1,\ldots, s$.
The fact that $\a(hA):\kY^s_\ell\to\kY^s_\ell$ and $\b(hA):\kY^s_\ell\to \kY_\ell$ are bounded and strongly continuous
in $h$  follows from the smoothness of $\a(hA) \P_m$ and $\b(hA) \P_m$  in $h$ which  is guaranteed by assumptions (EXP), (A1) and holds
for all $m>0$.  Moreover 
$\partial_h^k \a(hA) = A^k \a^{(k)}(hA)$ and so by \eqref{e.Y_ellCompare} we have
$\|\partial_h^k ( \a(hA)) \|_{\kY^s_\ell\to \kY^s_{\ell-k}}\leq M_\a^{(k)}$.
The same holds for $\b(hA)$.
We have  $ \partial_h( h\a(hA) ) = \a(hA) + hA \a'(hA)$
  and iteratively
  \[
   \partial_h^k( h\a(hA) ) = kA^{k-1} \a^{(k-1)}(hA) +hA^k \a^{(k)}(hA)
  \]
  and so \eqref{e.hahA} holds. The same applies for $\b(hA)$ 
which proves \eqref{e.hbhA}.
\end{proof}

\begin{theorem}[Regularity of numerical method]
  Assume   (A1), (A2) and (B), and apply an exponential integrator $\Psi$
 satisfying condition (EXP) to (\ref{eqn:see}).  Let $R>0$.
 Then there is $h_*>0$ such that   
  \begin{subequations}
\begin{equation}\label{eqn:glob-U-num_method_regularity}
    W^i(\cdot,h), \Psi(\cdot,h) \in \kCb^{N} (\kB_0^{R/2};\kB^R_0)
  \end{equation}
 for $i=1,\ldots, s$, 
with uniform bounds in  $h \in [0,h_*]$.
 Moreover, for
  $\ell \in I^-$, $k\in \N_0 $, $k \leq  \ell$,
 $i=1,\ldots, s$,
\begin{equation}\label{eqn:glob-num_method_regularity}
    W^i(U,\cdot), \Psi(U,\cdot) \in \kCb^k( [0,h_*];\kB_0^R)
  \end{equation}
\end{subequations}
  with uniform bounds in $U \in\kB^r_\ell$. 
  The bounds on  $\Psi$, $W$  and  $h_*$ 
  depend only on  $R$,
   the bounds  from   (B)   for $B$ and its derivatives on balls of radius $R$ and
 the bounds of the constants of the numerical method from (EXP).
  \label{thm:num_method_regularity}
\end{theorem}

\begin{proof} 
If the numerical method is explicit (i.e., $\a_{ij}=0$ for $i\leq j$) then we choose $h_*>0$ such that $h_* \max(M_\a, M_\b) M[R]\leq \frac{R}2$ noting \eqref{e.expInt} and (A2).
If the method is implicit, then, similarly
as in \cite{Wulff2016}, we   compute $W$ as
  fixed point of the map $\Pi:\kB^R_{\kY^s}(0) \times\kB^{R/2}_\kY(0) \times [0,h_*]\to
   \kY^s$, given by
 \begin{equation}
  \label{e.Pi-W}
  \Pi(W,U,h)  =\exp(h \c A) \1 U + h  \a (h A) B(W )
 \end{equation}
  using \eqref{e.Pi}.
 For $U \in \kB^{R/2}_\kY(0)$, $W \in \kB^R_{\kY^s}(0)$ we get from \eqref{e.normW}, (A2) and (B)
that
 \begin{align}
    \|\Pi(W,U,h) \|_{\kY^s} &\leq
\|\exp(h \c A) \1 U\|_{\kY^s} + h   \| \a (h A) \|_{\kY^s\to \kY^s} M[R] \notag\\
    &\leq    R/2 + h M_\a M[R]
    \leq R \label{e.Pi-Fix}
  \end{align}
  for $h\in [0,h_*]$ and $h_*$  as above.
So $\Pi$ maps $\kB^{R}_{\kY^s}(0)$ to itself.
Next choose $h_*$ such that also  $2h_*M_\a M'[R]\leq 1$. Then
$\|\D_W\Pi(W,U,h)\|_{\kY^s \to \kY^s} \leq hM_\a M'[R]\leq 1/2$  for $W\in \kB^R_{\kY^s}(0)$,   $h\in [0,h_*]$,  and  $\Pi$ is a
  contraction. Moreover $\Pi$ is continuous in $h$ by assumption (A1) and Lemma \ref{l.ahA}.
  Therefore, by (B),  $W\in
  \kCb(\kB^{R/2}_{\kY}(0) \times [0,h_*];\kB^{R}_{\kY^s}(0))$
with $N$ derivatives in $U$. 
 This proves 
\eqref{eqn:glob-U-num_method_regularity} and also
\eqref{eqn:glob-num_method_regularity}  in the case $k=0$  for $W$. 
If $k\neq 0$ then, since  $\ell \in I^-$,  the above argument also
  holds if $\kY$ is replaced by $\kY_{\ell-j}$, $j=0,\ldots,
  k$. Hence  there is some $h_*>0$  such that  $W^i  \in
  \kCb(\kB^{R/2}_{\ell-j}\times[0,h_*];\kB^R_{\ell-j})$, $j=0,\ldots,
  k$, $i=1,\ldots, s$.  The map $\Pi$ from   \eqref{e.Pi-W} belongs to the class of contraction mappings studied in \cite[Appendix]{OW10A}. Hence by the results of \cite[Appendix]{OW10A}  for $U \in\kB^{R/2}_{\ell}$ the 
$h$ derivatives up to order $k$ of $W$ can then be obtained by
  implicit differentiation of $\Pi(W,U,h)=W(U,h)$. Moreover the $h$ derivatives of $\Psi$ can be obtained  by differentiating \eqref{e.psi}, using Lemma \ref{l.ahA}.
\end{proof}
 
%%%%%%%%%%%%%%%%%%%%%%%%%%%%%%%%%%%

\section{Galerkin truncation of the exponential Runge Kutta method}

In this section we  truncate the semiflow $\Phi^t$ of
\eqref{eqn:see} and the numerical method $\Psi^h$ defined by
\eqref{e.expInt}   to a Galerkin subspace of
$\kY$ and study the truncation error.  Let $\phi_m^t(u^0_m)=u_m(t)$ be the flow of 
 the
projected  evolution equation
\begin{align}
  \frac{ \d u_m}{\d t}&= Au_m + B_m(u_m) \label{eqn:see_proj}
\end{align}
where $B_m(U) =\P_m B(\P_m U)$,
$A_m = \P_m A$ and define $\Phi^t_m :=\phi^t_m\circ \P_m$.
\eqref{eqn:inner_prod} implies the followig estimates which are crucial for our analysis: for $\ell \geq 0$, $k\geq 0$, $m\geq 1$,
\begin{equation}
  \norm{A^\ell\P_m U}{\kY}\leq m^\ell\norm{\P_m U}{\kY},\quad
 \|\P_m \|_{ \kY_\ell\to \kY_{\ell+k}} \leq m^k,
\quad \norm{\Q_m U}{\kY}\leq m^{-\ell}\norm{U}{\kY_\ell}.
  \label{eqn:proj_est}
\end{equation}

As shown in  \cite{Wulff2016} the following holds:

\begin{lemma}[Projection error for the semiflow]\label{l.see_proj_err}
 Assume  (A1), (A2) and (B) and  let  
 $\ell>0$. Then Theorem \ref{t.semiflow} applies to the  Galerkin truncated semiflow of  \eqref{eqn:see} uniformly in $m \geq 0$.
  Moreover for fixed $\delta>0$ and $T>0$
    and  for all $U^0$ with
 \begin{subequations}
 \begin{equation}\label{e.PhiGalErrorCond}
    \|\Phi^t(U^0)\|_{\kY_\ell} \leq R, ~~t \in [0,T]
  \end{equation}
  there is $m_*\geq 0$ such that for $m\geq m_*$ we have
  $\Phi_m^t(U^0) \in \kB^{R+\delta}_0$ for $t\in [0,T]$, and
  \begin{equation}\label{e.globPhiGalError}
    \| \Phi^t(U^0)-\Phi^t_m(U^0) \|_\kY = m^{-\ell}R  \e^{t M'}= \kO(m^{-\ell})
  \end{equation}
  for $m\geq m_*$ and
$t\in [0,T]$, where $M'=M'_0[R+\delta]$.
\end{subequations}
  Here $m_* $ and the order constant  depend only on $\delta$, $R$, $T$,
  \eqref{eqn:semigroup_bound} and the bounds afforded by
 (B) on balls of radius $R+\delta$.
\end{lemma}

Let $\psi^h_m(u_m^0)$ be the exponential Runge-Kutta integrator  $\Psi$ applied to the
projected semilinear evolution equation \eqref{eqn:see_proj} and  $w_m=w_m(u^0_m,h)$ its stage vector and  set $W_m(\cdot, h) =w_m(\P_m \cdot, h)$,
$\Psi^h_m=\psi^h_m\circ\P_m$. 
Similar to Lemma  
\ref{l.see_proj_err}, we have the following result:
 .

\begin{lemma}[Projection error of numerical method]
\label{l.reg-proj-num-method}
  Assume  (A1), (A2), (B), and   (EXP).  Let $R>0$.
 Then there is $h_*>0 $ such that Theorem \ref{thm:num_method_regularity} applies to  the stage vector $W_m$ and numerical method $\Psi_m$ of the projected
system \eqref{eqn:see_proj}  
with uniform bounds in   $m\geq 0$. 
Moreover,  if  $\ell \in I$, $\ell>0$,  then for
$m\geq 0$ we get 
\begin{subequations}
  \begin{equation}\label{e.wwm}
    \sup_{ \substack{  U\in\kB^{R/2}_\ell
        \\ h\in[0,h_*]  }    }
    \norm{W(U,h)-W_m(U,h)}{\kY^s}=\kO(m^{-\ell})
  \end{equation}
and
 \begin{equation}\label{e.psipsim}
    \sup_{ \substack{  U\in \kB^{R/2}_\ell
        \\ h\in[0,h_*]  }    }
    \norm{\Psi(U,h)-\Psi_m(U,h)}{\kY}=\kO(m^{-\ell}).
  \end{equation}
\end{subequations}
 The bounds on   $h_*$,  $\Psi_m$ and $W_m$   and the order constants 
depend only on $R$, the bounds  from   (B) on $B$ and its derivatives on balls of radius $R$ and
 on the bounds of the constants of the numerical method from (EXP).
\end{lemma}
  
 \begin{proof}
 The fact that Theorem \ref{thm:num_method_regularity} applies to the projected system uniformly in $m\geq 0$ is immediate from its proof. Moreover using (\ref{e.Pi})  we find
  \begin{align}
    \norm{W(U,h)-W_m(U,h)}{\kY^s} &\leq
    \norm{\exp(\c hA)}{\kY^s\to\kY^s}\norm{\Q_mU}{\kY}  + 
\norm{h \a (hA)  \Q_mB(W)}{\kY^s}  \notag\\
    & \quad 
+h\norm{\a(hA)}{\kY^s\to\kY^s} \norm{\P_m(B(W)-B(W_m)))}{\kY^s}  \notag\\
    &\leq   \norm{\Q_m U}{\kY}  + h M_\a   m^{-\ell}M_\ell[R]
     \notag \\
& \quad+ h  M_\a\norm{\P_m B(W(U,h))-\P_m B(W_m(U,h))}{\kY^s}
 \notag   \\
    & \leq  
    \norm{U}{\kY_\ell}m^{-\ell}+  hM_\a m^{-\ell}M_\ell[R]
   \notag  \\
    & \qquad+ h  M_\a M' \norm{W(U,h)-W_m(U,h)}{\kY^s}. \label{e.WWm-est}
  \end{align}
Here $M' = M'_0[R]$ and  
 we used (A1), (A2), (EXP)  and  \eqref{eqn:proj_est}.
 Hence
 \[
 \sup_{ \substack{  U\in\kB^{R/2}_\ell
        \\ h\in[0,h_*]  }    }
        \norm{W(U,h)-W_m(U,h)}{\kY^s} \leq   \frac{ m^{-\ell}R/2+  h_*M_\a m^{-\ell}M_\ell[R]}{1-h_*  M_\a M' }
 \]
 proving \eqref{e.wwm}. 
  For the numerical method using \eqref{e.psi}, Lemma \ref{l.ahA}, (A1), (A2) and \eqref{eqn:proj_est} we obtain 
  \begin{align}
   \norm{\Psi^h(U)-\Psi^h_m(U)}{\kY} & \leq  
 \norm{\exp(hA) }{\kY\to\kY}
 \norm{\Q_m U}{\kY}  +   \| \Q_m \b(hA) B(W)\|_{\kY^s} \notag \\
& \quad  
+ h \|\b(hA)\|_{\kY^s\to \kY} \norm{\P_m(B(W)- B(W_m))}{\kY^s} \notag \\
   & \leq  \norm{U}{\kY_\ell}m^{-\ell}  + 
h\|\b(hA)\|_{\kY^s\to \kY}   m^{-\ell} M_\ell[R] \notag   \\
&\quad +   h\norm{\b(hA)}{\kY^s\to \kY}  M'\norm{W(U)-W_m(U)}{\kY^s}
 \notag \\
   & \leq  \norm{U}{\kY_\ell}m^{-\ell}  +hM_\b  m^{-\ell} M_\ell[R]
\notag  +  hM_\b   M' \kO(m^{-\ell}).
\label{e.psipsimHelp}
  \end{align}
  Here we used \eqref{e.wwm} in the last line.
\end{proof} 
  Note that  an analogous result has been obtained, for A-stable Runge-Kutta methods,
in \cite[Lemma 4.3]{Wulff2016}.

%%%%%%%%%%%%%%%%%%%%%%%%%%%%%%%%%%%%%%%%%%%%%%%%%%%%%%%%%%%%%%%%%%%%%%%%%%%%%%%%%%%%%%%%%%%%%%%%%%%

\section{Trajectory error bounds for non-smooth data}
\label{sec:traj}
 
In this section we prove our main result for exponential Runge-Kutta methods:

\begin{theorem}[Trajectory error for nonsmooth data]
 \label{thm:expInt_convergence}
  Assume 
  (A1), (A2) and (B) and apply an exponential Runge-Kutta method \eqref{e.expInt} subject to 
  (EXP) to  \eqref{eqn:see}. Let
 $\ell \in I^-$, $\ell >0 $, and fix $T>0$ and $R>0$. Then there exist constants
  $h_*>0$, $c_1>0$, $c_2>0$  such that for every   $U^0 $
 with
  \begin{equation}\label{e.condPhiGlobal}
   \|\Phi^t(U^0)\|_{\kY_\ell} \leq R, \quad\mbox{for}~~t\in[0,T]
  \end{equation}
  and for all
  $h\in[0,h_*]$   we have
  \begin{equation}
  \| \Phi^{nh}(U^0) -( \Psi^h)^n(U^0)\|_\kY \leq c_1e^{c_2nh}h^{\min(\ell,p)},
 \label{e.globTrajEst}
 \end{equation}
 provided that $nh\leq T$.
 The constants $h_*$,
  $c_1$ and $c_2$ depend only on  $R$, $T$,  the bounds from (EXP) on  the numerical method
  and the bounds afforded by (B).
\end{theorem}

Before we prove this result we provide an example of a linear evolution equation and some numerical evidence
where the estimates are sharp.

\begin{example} \rm
Applying the exponential Euler method  to a linear evolution equation
$\dot U = AU + B U$ with $A$ skew-symmetric and $[A,B]=0$ gives
\[
U(nh) - U^n = \e^{nhA}(\e^{nhB} - (\id + h C(hA))^n)U^0.
\]
where $C(hA) = \varphi_1(-hA)B$. Assume that ${\rm spec}(A) = \i \Z$ and let $A e_k=\i k e_k$, $B e_k = \lambda_k e_k$, $\|e_k\|_\kY=1$. Let $h = \pi/k$, $n=k$.  
Then we compute that $C(hA) e_k = -\frac{2\i}{\pi} \lambda_k e_k$ and so 
 \[
  (\id + h C(hA))^n)e_k = 
  (\e^{-2\i\lambda_k}+ O(h) )e_k 
 \]
 which is not close to $\e^{nhB} e_k = \e^{\pi\lambda_k} e_k $, so the global error is $O(1)$ for $U^0=e_k$.
 If we choose $U =O(k^{-\ell}) e_k$ then $\|\varphi_1(-hA)U -U\|_\kY = O(k h k^{-\ell}) = O(h^\ell) $ and so
 we get an convergence error of order $h^\ell$. So the estimate of Theorem \ref{thm:expInt_convergence} 
 is sharp in this case.
 \label{ex.EstSharp}
\end{example}

\begin{figure}[h]
  \begin{center}
  \vspace*{-6cm}
% \vspace*{-5.5cm}
% \hspace*{-0.9cm}
  \includegraphics[scale=0.6]{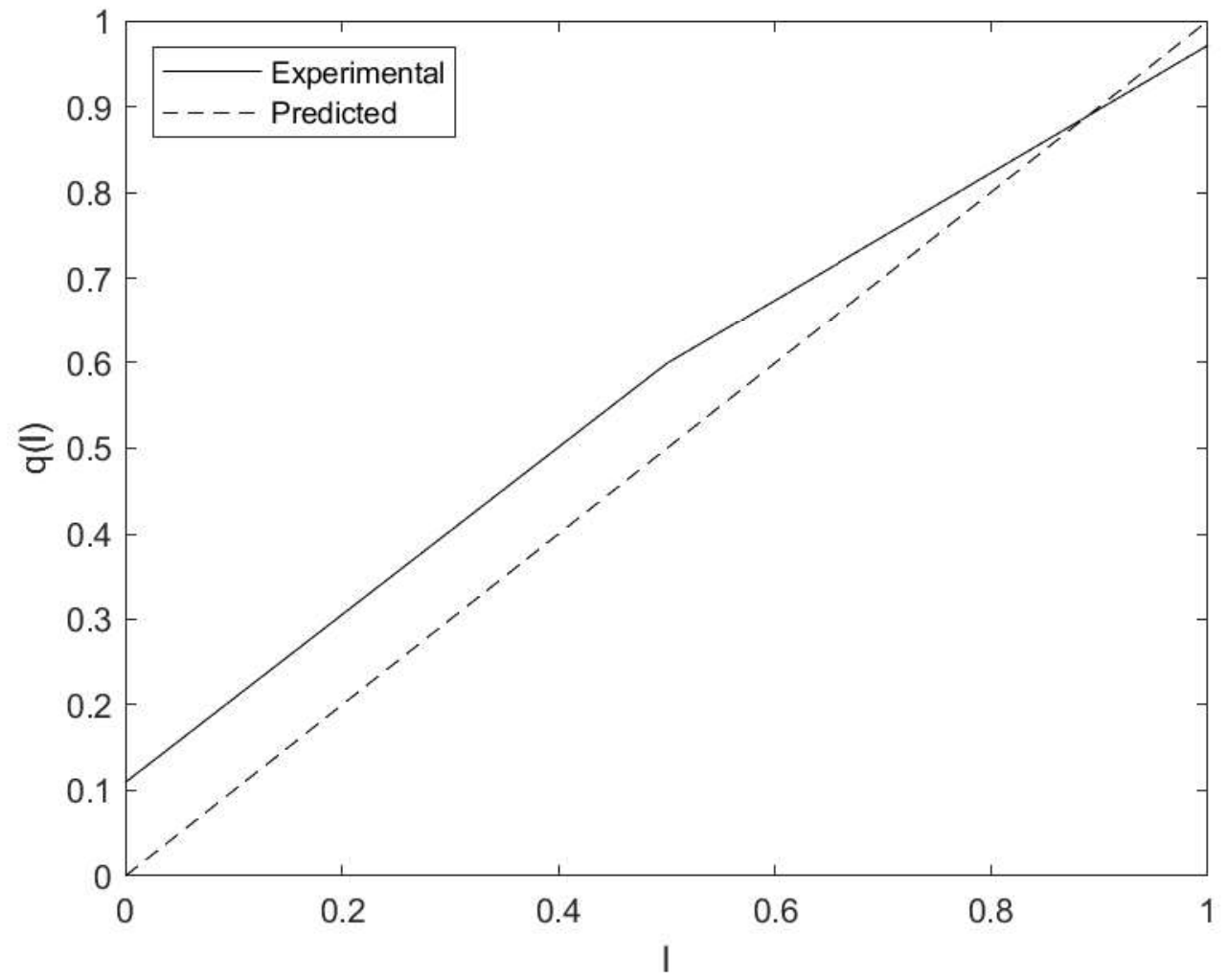}
  \end{center}
    \vspace*{-6cm}
%\vspace*{-4.9cm}
  \caption{Plot of a numerical estimate of $q(\ell)$ against $\ell$
    for the exponntial Euler method applied to the semilinear wave
    equation, with the prediction of Theorem~\ref{thm:expInt_convergence}
    for comparison.}
  \label{fig:expEuler_swe}
\end{figure}
 
 \begin{example}\rm
In  Figure \ref{fig:expEuler_swe} we display  the order of
convergence of the exponential Euler method  (see Example
\ref{ex.expRK}) which has classical order $p=1$ applied to the semilinear
wave equation $u_{tt} = u_{xx} - V'(u)$, $x\in [0,2\pi]$, with periodic boundary conditions and  $V'(u)=u-4u^2$ for   $\ell=j/2$, $j=0,\ldots, 6$, on the integration
interval $t \in [0,0.5]$,
using a fine spatial mesh (we use $N=1000$ grid points on $[0, 2\pi]$). 
As in \cite{Wulff2016}  we choose the   initial values $U^0= (u^0, v^0) \in \kY_\ell$ where
\[
u^0(x) =
 \sum_{k=0}^{N=1} \frac{c_u}{k^{\ell + 1/2+ \epsilon}} (\cos kx
+   \sin kx),
\quad
v^0(x)  =
 \sum_{k=0}^{N=1} \frac{c_v}{k^{\ell + 1/2+ \epsilon}} (\cos kx
+ \sin kx).
\]
Here $c_u$ and $c_v$ are such that $\|U^0\|_{\kY_\ell}= 1$, with $U^0 = (u^0, v^0)$, 
and  $\epsilon = 10^{-8}$.  
%We   integrate the semilinear wave equation with the above initial data for the time steps  $h = 0.1, 0.095, 0.09, 0.085,\ldots, 0.05$,  when $\ell>0$. At $\ell=0$, to reduce computational effort, we only used   the time steps  $h = 0.1,  0.09, \ldots 0.05$. To estimate the trajectory error,  we compare the numerical solution  to a solutioncalculated using a much smaller time step, $\tilde h = 10^{-3}$ for $\ell>0$ and $\tilde h = 10^{-4}$ for $\ell=0$.  From the assumption $E_n(h) = c h^{q}$ we get $\log E_n(h) = \log c + q \log h$. Fitting a line to those data, we take the gradient of the line as our estimated order of convergence of the trajectory error.
Approximating the order of convergence of the method  numerically as in \cite{Wulff2016}
we see that the numerical data confirm the theoretically predicted order of convergence
$q(\ell) = \ell$ for initial data in $\kY_\ell$, $\ell \leq p$ of Theorem \ref{thm:expInt_convergence}.
Note that the order of convergence does not decrease to exactly $0$ at $\ell=0$   because we simulate a space-time discretization rather than a time semidiscretization.  

\end{example}

To prove Theorem \ref{thm:expInt_convergence} we  analyze the dependence on $m$ of the local
error of an exponential Runge-Kutta method \eqref{thm:expInt_convergence} applied to the Galerkiin truncated equation (\ref{eqn:see_proj}) for low regularity initial data under the assumptions (A1), (A2), (B) and (EXP). 
As in \cite{Wulff2016} for $A$-stable Runge-Kutta methods, by coupling $m$ and $h$ and balancing the Galerkin truncation error and 
trajectory error of the Galerkin truncated system, we  prove our convergence result, Theorem \ref{thm:expInt_convergence}.

 %%%%%%%%%%%%%%%%%%%%%%%%%%%%%%%%
 %%%%%%%%%%%%%%%%%%%%%%%%%%%%%%%%%

\subsection{Some lemmas}\label{ss.Prelim}
We first present some  lemmas that will be needed in the proof.
Let $\hat\Phi^t_m = \Phi^t_m-\e^{tA_m}\P_m $, where $A_m = \P_m A$. In the following let $[a]_+ =  \max(a,0)$ for $a\in \R$.
 
\begin{lemma}[$m$-dependent bounds for derivatives of $\Phi_m$ and $\hat\Phi_m$]\label{l.phimDerivs}
Assume   (A1), (A2) and (B) and 
choose 
$\ell\in I^-$, $T>0$ and $R>0$.
 Then  for all $U^0$  satisfying 
\begin{subequations}
 \begin{equation}\label{e.glob-cond-proj_flow_reg}
  \Phi^t_m(U^0) \in \kB_\ell^R\quad\mbox{for}\quad t\in [0,T],~~m\geq m_*,
  \end{equation}
  %\eqref{e.glob-cond-proj_flow_reg}
 and for all $k\in \N$, $N\geq k$, $\Phi_m(U^0) \in  \kC_b^{k}( [0,T];\kB_0^R)   $ with  $m$ dependent bounds on the derivatives that satisfy
 \begin{equation}   \norm{\partial_t^k\Phi^t_m(U^0)}{\kCb([0,T];\kY )}=  \kO(m^{[k-\ell]_+}) \label{eqn:proj_flow_m_dep} 
 \end{equation}
 and
\begin{equation}\label{e.hatphimDerivs}
\|\partial^k_t \hat\Phi_m^t(U^0)\|_\kY= O(m^{[k-\ell-1]_+}) + O(t m^{[k-\ell]_+}).
\end{equation}
\end{subequations}
 The  order
constants only depend on $T$,    $R$, \eqref{eqn:semigroup_bound}
and the bounds from  (B).
\end{lemma}

\begin{proof}  \eqref{eqn:proj_flow_m_dep}  is shown in \cite[Lemma 5.1]{Wulff2016}.
To prove \eqref{e.hatphimDerivs} note that we have
\begin{align*}
\partial_t \hat\Phi^t_m(U^0) & =   A_m \Phi^t_m(U^0) + B_m(\Phi^t_m(U^0)) -A_m\e^{tA_m}\P_m U^0 \\
&  =   A_m\hat\Phi^t_m(U^0)  + B_m(\Phi^t_m(U^0)).
\end{align*}
Similarly
\[
  \partial_t^2\hat\Phi^t_m(U^0) =
   A_m^2\hat\Phi^t_m(U^0) + 
    \frac{\d}{\d t} \left(  B_m(\Phi^t_m(U^0)) \right) + 
A_m B_m(\Phi^t_m(U^0))   
\]
and more generally
\[
\partial^k_t \hat\Phi^t_m(U^0) =  A_m^k \hat\Phi^t_m(U^0)
+  \sum_{j=0}^{k-1} A_m^{k-1-j} \frac{\d^j}{\d t^j}\left( B_m(\Phi^t_m(U))\right) .
\]
Hence   for $U^0 \in \kY_\ell$ by \eqref{e.Y_ellCompare}
\begin{align}\label{e.k_deriv_hatPhim}
\|\partial^k_t \hat\Phi^t_m(U^0)\|_{\kY} 
& \leq  \| \hat\Phi^t_m(U^0) \|_{\kY_k}
+  \sum_{j=0}^{k-1} \| \frac{\d^j}{\d t^j} \left( B_m(\Phi^t_m(U^0))\right)\|_{\kY_{k-1-j}}
\end{align}
Moreover
\[
\hat\Phi^t_m(U^0) = \int_0^t \e^{(t-s) A_m} B_m(\Phi^s_m(U^0))\d s 
\]
and so by (A2) and (B)
\begin{align}
\| \hat\Phi^t_m(U^0) \|_{\kY_k}  &\leq  t\max_{s \in [0,t]} \| B_m(\Phi^s_m(U^0)) \|_{\Y_k}
\leq t \max_{s \in [0,t]} m^{[k-\ell]_+} \| B_m(\Phi^s_m(U^0)) \|_{\Y_\ell}\notag \\
& \leq t m^{[k-\ell]_+} M_{\min(\ell,k)}[R].
\label{e.hatPhi_kNorm}
\end{align}
Using the Fa\`a di Bruno
  formula \cite{FaadiBruno} we find that for any $i \in \N$,  $i\leq N$, with
  $u_m = \Phi^t_m(U^0) \in \kB_\ell^R$,
  \begin{align}
     \frac{\d^i}{\d t^i} \left( B_m(u_m) \right)
    &=   \sum_{1\leq \beta \leq i}\frac{i!
      \D^\beta_uB_m(u_m)}{j_1!\cdots j_i!} 
    \prod_{\alpha=1}^i\left(\frac{\partial_t^\alpha
        u_m}{\alpha!}\right)^{j_\alpha},
    \label{e.FaaDiBruno}
  \end{align}
  where  
 $\beta=j_1+\cdots+j_i$ and the    sum is over all $j_\alpha \in \N_0$, $\alpha = 1,\ldots, i$, with 
  $j_1+2j_2+\cdots +  ij_i=i$.
Using \eqref{eqn:proj_flow_m_dep} and the Faa di Bruno formula we then get 
\begin{align*}
 \left\|   \frac{i!
      \D^\beta_uB_m(u_m)}{j_1!\cdots j_i!} 
    \prod_{\alpha=1}^i\left(\frac{\partial_t^\alpha
        u_m}{\alpha!}\right)^{j_\alpha}  \right\|_\kY= \kO(m^n)
 \end{align*}
 with 
\begin{align}\label{e.n}
n= (\lceil \ell \rceil -\ell)j_{\lceil \ell \rceil}+\cdots+
(i-\ell)j_i\leq [i-\ell]_+.
\end{align}
Hence  for $N\geq i$, $\ell\geq 0$
\[
\| \frac{\d^i}{\d t^i} \left(  B_m(u_m)\right) \|_{\kY} =O(m^{[i-\ell]_+}).
 \]
Replacing $\kY$ by $\kY_{k-1-j}$, $\ell$ by $\ell-(k-1-j)$ and
 $i$ by $j$ we see
that  for $U \in \kY_\ell$
\[
\|\frac{\d^j}{\d t^j} \left(  B_m(u_m)\right) \|_{\kY_{k-1-j}}  = O(m^{[j - (\ell-(k-1-j))]_+})=
O(m^{[k-\ell-1]_+}).
\]
Plugging this into \eqref{e.k_deriv_hatPhim} and using  \eqref{e.hatPhi_kNorm} shows \eqref{e.hatphimDerivs}.
\end{proof}

Similar to Lemma \cite[Lemma 5.2]{Wulff2016} for $A$-stable Runge Kutta methods we have:

\begin{lemma}[$m$-dependent bounds for derivatives of   $W_m$]
  Assume  (A1), (A2) and (B), and apply an exponential Runge-Kutta method $\Psi$ 
  satisfying (EXP) to \eqref{eqn:see}.  Let $R>0$ and 
 $\ell\in I^-$ and
$k\in \N_0$ with 
  $\ell \leq k\leq N$.
Then there is
  $h_*>0$   such that  for $m\geq 0$
and    $i=1,\ldots, s$, 
\begin{subequations}
  \begin{equation}
    W_m^i(U,\cdot) \in \kC_b^{k}([0,h_*];\kB_0^R)
    \quad\mbox{
      for}   \quad i=1,\ldots,s
    \label{e.wmpsim}
  \end{equation}
  with $m$-dependent bounds which are uniform in 
$U \in \kB_\ell^{R/2}$.
Moreover 
 \begin{equation}\label{e.deriv-wm-bound}
    \sup_{ \substack{  U\in \kB_\ell^{R/2}
        \\ h\in[0,h_*]  }    }\norm{\partial_h^{k}W_m(U,h)}{\kY^s} = \kO(m^{k-\ell})
  \end{equation}
and   
  \begin{equation}\label{e.derivBmbound}
    \|\partial_h^j \P_m B(W_m(U,h))\|_{\kY^s} = \kO(m^{j-\ell})
  \end{equation}
  \end{subequations}
  for all $h\in [0,h_*]$, $U \in\kB^{R/2}_\ell$ and $k\geq j
  \geq \ell$.
The  order constants in  \eqref{e.deriv-wm-bound} and \eqref{e.derivBmbound} depend only $R$, the bounds of the constants of the numerical method from (EXP)
and the bounds from (B)  for $B$ and its  derivatives on balls of radius $R$. 
 \label{lem:wm-m-dep-deriv}
\end{lemma}

\begin{proof}
To prove \eqref{e.deriv-wm-bound}, 
 differentiate \eqref{e.Pi} $k$ times in
  $h$: 
  \begin{equation}\label{e.derivwmInduction}
    \partial_h^k W_m = \partial_h^{k}(\exp(\c h A))\P_m U\1
    + 
    \sum_{j=0}^k{k \choose j}\partial_h^{k-j}(h\a( h  A) \P_m) \partial_h^j B(W_m).
  \end{equation}
 By  \eqref{eqn:proj_est}, 
  for $k\geq \ell$,
  \begin{align}
  \sup_{h\in [0,h_*]}
  \| \partial_h^k(\exp(\c h  A) \P_m \|_{\kY_\ell^s\to \kY^s}  &=
  \sup_{h\in [0,h_*]} \|  \c^k A^k \exp(\c h  A) \P_m \|_{\kY_\ell^s\to \kY^s} \notag\\
  & \leq  \|  \c^k\|_{\kY^s \to \kY^s}  \|\P_m \|_{\kY_\ell^s\to \kY^s_k} =
   O( m^{k-\ell}).
   \label{e.firstterm}
  \end{align}
Next for $0\leq j\leq \ell \leq k$
 \begin{align}
\|\partial^{k-j}_h ( h\a (h A))\P_m \|_{ \kY^s_{\ell-j}\to\kY^s}& \leq
\|\partial^{k-j}_h ( h\a (h A)) \|_{ \kY^s_{k-j}\to\kY^s} \|\P_m \|_{ \kY^s_{\ell-j}\to\kY_{k-j}^s}\notag\\
& \leq M_{h,\a}^{(k-j)} \|\P_m \|_{ \kY^s_{\ell-j}\to\kY_{k-j}^s}  = M_{h,\a}^{(k-j)} m^{k-\ell}.
\label{e.k-jDerivha}
 \end{align}
where we used  \eqref{e.hahA} (with $k$ replaced by $k-j$) and \eqref{eqn:proj_est}.
Moreover the regularity of the stage vector on $\kY_{\ell-j}^s$  (see  Theorem \ref{thm:num_method_regularity} and Lemma
  \ref{l.reg-proj-num-method})  gives 
  \begin{align}
     W_m^i(U,\cdot) \in \kC_b^{j} ([0,h_*]; \kB^R_{\ell-j}),
   \label{eqn:proj_wm_indep}
  \end{align}
for $i=1,\ldots, s$, $j=1\ldots, \lfloor\ell\rfloor$,
  with bounds uniform in $m\geq 0$ and 
$U \in \kB^{R/2}_\ell$.   
Using these two estimates 
we can obtain for the $j$-th term in the
sum of \eqref{e.derivwmInduction} for $0\leq j\leq \ell \leq k$ and
$h\in [0,h_*]$ the following:
\begin{align}
\| &\partial^{k-j}_h ( h \a(hA)) \partial_h^j \P_m
  B(W_m(U,h))\|_{\kY^s}\notag \\
& \leq
\|\partial^{k-j}_h ( h(\a (h A))\P_m \|_{ \kY^s_{\ell-j}\to\kY^s} \| \partial_h^j 
  B(W_m(U,h))\|_{\kY_{\ell-j}^s} \notag\\
& \leq M_{h,\a}^{(k-j)} m^{ k-\ell}   \| \partial_h^j 
  B(W_m(U,h))\|_{\kY_{\ell-j}^s} \leq O(m^{ k-\ell}).
\label{e.jLessEll-W}
\end{align}
To estimate the $j$th term in the sum of \eqref{e.derivwmInduction}  for $j>\ell$ and hence  prove \eqref{e.deriv-wm-bound} and \eqref{e.derivBmbound}  
 we proceed inductively for  $k = \lceil \ell \rceil,\ldots,N$.  If $\ell \in \N_0$ then the start
  of the induction is $k = \ell$, and \eqref{e.deriv-wm-bound} and \eqref{e.derivBmbound}  
  follow from Theorem~\ref{l.reg-proj-num-method}. If $\ell \notin \N_0$, then the
  start of the induction is $k = \lceil \ell \rceil > \ell$.  
  If $k=  \lceil \ell \rceil$ then
 the first term in  \eqref{e.derivwmInduction}  is of order
  $\kO(m^{k-\ell})$ by \eqref{e.firstterm}, and all  terms in the sum  of \eqref{e.derivwmInduction} are  $\kO(m^{k-\ell})$ for all $h\in[0,h_*]$ due to
 \eqref{e.jLessEll-W}   
  except for the last term. Hence,  
  \begin{equation}\label{e.derivwmest}
    \sup_{\substack{h\in [0,h_*]\\ U \in \kB^{R/2}_\ell  }}
    \| \partial_h^k W_m(U,h)\|_{\kY^s} \leq \kO(m^{k-\ell}) + 
    h_* M_\a\sup_{\substack{h\in [0,h_*]\\ U \in  \kB^{R/2}_\ell }}\| \partial_h^k B(W_m(U,h))\|_{\kY^s}.
  \end{equation}
  The Fa\`a di Bruno formula \eqref{e.FaaDiBruno} gives
  \begin{equation}\label{e.Bwm}
    \partial_h^k B(W_m(U,h)) = 
    \sum_{1\leq \beta \leq k}\frac{k! \D^\beta_wB_m(W_m(U,h))}{j_1!\cdots j_k!}
      \prod_{\alpha=1}^k\left(\frac{\partial_h^\alpha W_m(U,h)}{\alpha!}\right)^{j_\alpha}
  \end{equation}
  where  $\beta=j_1+\cdots+j_k$ and the  sum is over all $j_\alpha \in N_0$, $\alpha=1,\ldots, k$ with
  $j_1+2j_2+\cdots+ kj_k=k$. All terms in the sum in  \eqref{e.Bwm} contain $h$-derivatives of order at most $k-1$
  and are therefore bounded  independent of $m$
  except when $\beta=j_k=1$ and $j_\alpha=0$ for $\alpha\neq k$.  Therefore
  \begin{align}
    \sup_{\substack{h\in [0,h_*]\\ U \in \kB^{R/2}_\ell   }}
    \| \partial_h^k B(W_m(U,h))\|_{\kY^s} &
  %   \leq \kO(m^{k-\ell}) +   \sup_{\substack{h\in [0,h_*]\\ U \in \kB^{R/2}_\ell }}   \|\D B(W_m(U,h)) \partial_h^k W_m(U,h)\|_{\kY^s} \notag \\
  %
    & \leq \kO(m^{k-\ell}) + M'_0[R] \sup_{\substack{h\in [0,h_*]\\ U \in
        \kB^r_\ell  }} \| \partial_h^k W_m(U,h)\|_{\kY^s}.
    \label{e.Bwmderivest}
  \end{align}
  Plugging this into \eqref{e.derivwmest} gives 
  \eqref{e.deriv-wm-bound} for $k = \lceil \ell \rceil$ and $h_*$ small
  enough. This estimate and \eqref{e.Bwmderivest}
  also shows \eqref{e.derivBmbound} for $k=\lceil \ell \rceil$.

  Now assume these estimates hold true for all $\hat k \in \N_0$ with $\ell \leq \hat{k}\leq
  k-1$ and let $k\leq N$. Then the first term in    \eqref{e.derivwmInduction} is
  $\kO(m^{k-\ell})$ by \eqref{e.firstterm} and by \eqref{e.jLessEll-W} the terms in the sum  of \eqref{e.derivwmInduction} with $\ell\geq j$ are  $\kO(m^{k-\ell})$ as well. For $k>j> \ell$ we estimate
  \begin{align}
\| &\partial^{k-j}_h ( h \a(hA)) \partial_h^j \P_m
  B(W_m(U,h))\|_{\kY^s}\notag \\
& \leq
\|\partial^{k-j}_h ( h(\a(h A))\P_m \|_{ \kY^s \to\kY^s} \| \partial_h^j 
  B(W_m(U,h))\|_{\kY^s} \notag \\
  &  \leq M_{h,\a}^{(k-j)} m^{k-j} \| \partial_h^j 
  B(W_m(U,h))\|_{\kY^s}  \leq O(m^{k-j}) \kO(m^{j-\ell}) = O(m^{ k-\ell}) 
\label{e.jGreaterEll-W}
\end{align}
where we used \eqref{e.k-jDerivha} and the induction hypothesis \eqref{e.derivBmbound} for $j<k$. Therefore under the induction hypothesis  all
  terms in   \eqref{e.derivwmInduction} are
  $\kO(m^{k-\ell})$ except from the last term in the sum, hence
  \eqref{e.derivwmest} holds true under the induction
  hypothesis, and so,
 each term in the sum of  the Fa\`a di Bruno formula \eqref{e.Bwm} with 
$j_k=0$ is of order
$\kO(m^n)$  in the $\kY^s$ norm  with 
$n  \leq k-\ell$ as in \eqref{e.n} (with $i$ replaced by $k$).
 Hence \eqref{e.Bwmderivest}  remains valid, and from  \eqref{e.derivwmest} we deduce
\eqref{e.deriv-wm-bound}  and   \eqref{e.derivBmbound}.
\end{proof}

  %%%%%%%%%%%%%%%%%%%%

 Let $\hat\Psi_m =\Psi_m- \e^{hA_m}$.
\begin{lemma}[$m$-dependent bounds for derivatives of $\hat\Psi_m$]
  Assume  (A1), (A2) and (B), and apply an exponential Runge-Kutta method $\Psi$
  satisfying (EXP) to  \eqref{eqn:see}. Choose 
 $\ell\in I^-$ and
$k\in \N_0$ with 
  $\ell \leq k\leq N$.
  Let $R>0$.
Then there is
  $h_*>0$   such that  for $m\geq 0$
  \begin{equation}
     \Psi_m(U,\cdot) \in \kC_b^{k}([0,h_*];\kB_0^R)
    \label{e.psim}
  \end{equation}
  with $m$-dependent bounds which are uniform in 
$U \in \kB_\ell^{R/2} $.
Moreover 
\begin{equation}\label{e.deriv-hatpsim-bound}
    \sup_{ \substack{  U\in\kB_\ell^{R/2}    \\ h\in[0,h_*] }}
\norm{\partial_h^{k}\hat\Psi^h_m(U)}{\kY} = \kO(hm^{k-\ell}) +\kO(m^{[k-\ell-1]_+}) .
 \end{equation}
The  order constants in  \eqref{e.deriv-hatpsim-bound} depend only $R$, the bounds on the numerical method from (EXP) 
and the bounds afforded by (B) on balls of radius $R$. 
 \label{lem:psim-m-dep-deriv}
\end{lemma}

\begin{proof}
 From \eqref{e.psi} we
formally  obtain 
  \begin{equation}\label{e.dkhPsi_m}
 \partial_h^k \hat\Psi_m = \sum_{j=0}^k
{k \choose j}    \partial^{k-j}_h (h \b^T(Ah)) \partial_h^j \P_m
  B(W_m(U,h)).
  \end{equation}
 From  \eqref{e.hbhA}  and \eqref{eqn:proj_est} we obtain for $n\in \N$, $n\geq \ell$,
 \begin{align}
\|\partial^{n}_h  ( h\b^T (h A)\P_m  \|_{\kY^s_\ell \to \kY} 
& \leq  
n M_\b^{(n-1)} \|\P_m \|_{\kY^s_{\ell}\to \kY^s_{n-1} }+   h M_\b^{(n)} \|\P_m\|_{\kY^s_{\ell} \to \kY^s_{n}}   \notag\\
& \leq   O(hm^{n-\ell})  +O(m^{[n-\ell-1]_+}).
 \label{e.k-jDerivhb}
 \end{align}
Using \eqref{e.k-jDerivhb}  (with $n=k-j$ and $\ell$ replaced by $\ell-j$) and
\eqref{eqn:proj_wm_indep},
we can estimate the $j$-th term in the
sum of \eqref{e.dkhPsi_m} for $ j\leq \ell$, $j< k$,  as follows:
\begin{align}
\| &\partial^{k-j}_h ( h \b^T(hA)) \partial_h^j \P_m
  B(W_m(U,h))\|_{\kY}\notag \\
& \leq
\|\partial^{k-j}_h ( h(\b^T(h A))\P_m \|_{ \kY^s_{\ell-j}\to\kY} \| \partial_h^j 
  B(W_m(U,h))\|_{\kY_{\ell-j}^s} \notag\\
&  \leq  O(hm^{ k-\ell}) +O(m^{[ k-\ell-1]_+}).
\label{e.jLessEll}
\end{align}
 Using \eqref{e.k-jDerivhb} and \eqref{e.derivBmbound}
 we can estimate the $j$-th term in the
sum of \eqref{e.dkhPsi_m} for $k> j\geq \ell$, $U \in \kB_\ell^{R/2}$, as follows:
\begin{align}
 \|\partial^{k-j}_h  & ( h\b^T(h A)) \partial_h^j \P_m
  B(W_m(U,h))\|_{\kY}\notag \\
& \leq
\|\partial^{k-j}_h ( h \b^T(hA))\P_m \|_{ \kY^s\to \kY} \| \partial_h^j 
  B(W_m(U,h))\|_{\kY^s}\notag\\
& =(\kO(h m^{k-j}) + \kO(m^{k-j-1} )) \kO( m^{j-\ell})
= O(hm^{k-\ell}) + O(m^{[k-\ell-1]_+}).
\label{e.j>Ell}
\end{align}
When $j=k$ then the second term in the estimates of \eqref{e.jLessEll} and \eqref{e.j>Ell}
disappears.
 These estimates together with \eqref{e.dkhPsi_m} then prove \eqref{e.deriv-hatpsim-bound}.
\end{proof}

  %%%%%%%%%%%%%%%%%%%%%%%%%%%%%%%%%%%%%%%%

%\subsection{Trajectory error for nonsmooth data }\label{ss.err_nonsmooth}
\subsection{ Proof of Theorem \ref{thm:expInt_convergence} }\label{ss.err_nonsmooth}
The proof has the same format as in \cite{Wulff2016}. To estimate the error of the  time semi-discretization
 we first discretize in space by a Galerkin truncation and prove regularity of the solution  of the truncated system. 
Then we estimate the error of the time discretization
of the space-discretized system  and couple the spatial discretization parameter $m$ with the time step size $h$. Finally we  prove regularity of the space-time discretization to estimate the  truncation error of the  time discretization. 
 
 We assume without loss of generality that $\ell\leq p$ noting \eqref{e.Y_ellCompare}, i.e., we replace $\ell$ by $ \min(\ell,p)$. 
\\

\noindent
{\em Step 1 (Regularity of solution of the Galerkin truncated system)}
This step is identical to \cite{Wulff2016}.  We include it for sake of completeness.
%In a first step we  aim to prove regularity  of the continuous solution of the projected system
%$u_m(t) = \phi_m^t(\P_m U^0) =\Phi_m^t(  U^0)$ which will be needed later.
We denote $R$ from \eqref{e.condPhiGlobal} as $R_\Phi$ to indicate that it is a bound on $\Phi^t(U^0)$. Then we have 
\begin{align}
\|\Phi_m^t(U^0)) \|_{\kY_\ell}  & \leq \|\P_m \Phi^t(U^0) -\Phi_m^t(U^0) \|_{\kY_\ell} + \|\P_m \Phi^t(U^0) \|_{\kY_\ell} \notag \\
& \leq m^\ell \|\P_m \Phi^t(U^0) -\Phi_m^t(U^0) \|_{\kY} 
+  \| \Phi^t(U^0) \|_{\kY_\ell} \notag \\
& \leq R_\Phi  \e^{ M'_0[2R_\Phi] T}  + R_\Phi  =:  r_\phi 
\label{e.hatC_ell}
\end{align}
for $U^0$ satisfying \eqref{e.condPhiGlobal},  $t\in [0,T]$ and $m\geq m_*$, where $m_*\geq 0$ is sufficiently
large. Here   we used 
\eqref{eqn:proj_est} in the second estimate and  Lemma \ref{l.see_proj_err} with  $\delta = R_\Phi$ and  \eqref{e.condPhiGlobal} 
in the final estimate.  

\smallskip
\noindent
{\em Step 2 (Trajectory error of the space time discretization)}
 Next we estimate the global error of the space time discretization, for $jh\leq T$,
\begin{equation}
  E_m^j(U^0,h)=\norm{\Phi^{jh}_m(U^0)-(\Psi^h_m)^j(U^0)}{\kY}.
  \label{eqn:ge_proj}
\end{equation}
  Using   \eqref{e.hatC_ell} for any $U^0$ satisfying \eqref{e.condPhiGlobal} and  all $(n+1)h\leq T$, $h\in [0,h_*]$, $m\geq m_*$
  we have
    \begin{align}
      E_m^{n+1}(&  U^0, h) = \norm{\Phi^{(n+1)h}_m( U^0) -
        (\Psi^h_m)^{n+1}(U^0) }{\kY}
      \notag \\
      &\leq \norm{ \Phi^h_m(\Phi_m^{nh}(U^0)) - \Psi^h_m(\Phi_m^{nh}(U^0))
      }{\kY}
    + \norm{ \Psi^h_m(\Phi_m^{nh}(U^0)) -
        \Psi^h_m((\Psi^h_m)^{n}(U^0)) }{\kY}
      \notag \\
      &\leq \frac{h^{p+1}}{(p+1)!}\left(\sup_{ t\in[0,T ]}
        \norm{\partial_h^{p+1}\hat\Phi^h_m(\Phi_m^t(U^0))}{\kY}+
     \sup_{ t\in[0,T] }   \norm{\partial_h^{p+1}\hat\Psi^h_m(\Phi_m^t(U^0))}{\kY} \right) \notag \\
      &\quad +
      \sup_{U\in \kB^{ 2r_\phi}_0}\norm{\D\Psi^h_m(U)}{ \kY\to \kY}\cdot
      E_m^n(U^0,h)\label{e.helpMain}
    \end{align}
 provided that
\begin{equation}\label{e.CondmeanvalueThm}
E_m^n(U,h)\leq r_\phi, \quad nh \leq T, h\in [0,h_*].
\end{equation}
The first term in \eqref{e.helpMain} is $h O(m^{p+1-\ell }) + O(m^{p-\ell}) $   by  Lemma \ref{l.phimDerivs}, with $R$ replaced by $r_\phi$ and  Lemma~\ref{lem:psim-m-dep-deriv}, with   $R$ replaced by $2r_\phi$, respectively. To bound  the second term note  that by Lemma~\ref{l.reg-proj-num-method}  (with  $R$ replaced by $4 r_\phi $ in \eqref{eqn:glob-U-num_method_regularity})
there is $h_*>0$ such that  $
W^i_h, \Psi^h_m \in \kCb^1(\kB^{2r_\phi}_0;\kB^{4 r_\phi}_0)
$, $i=1,\ldots, s$,  for $m \geq 0$,  $h \in [0,h_*]$
with uniform bounds in $m \geq 0$,  $h \in [0,h_*]$.
Then, using \eqref{e.psi}, \eqref{eqn:semigroup_bound}, (A2) and (EXP)
  we get  for $h \in [0,h_*]$
\begin{align}
 \| \D_U\Psi_m^h(U)\|_{\kY\to\kY}  & \leq   \|\exp(hA)\|_{ \kY\to \kY} + h \|\b(hA)\|_{\kY^s\to \kY} M'_0[4 r_\phi ]  \| \D_U W_m(U)\|_{ \kY\to\kY^s}
\notag \\
& \leq   1   + h M_\b M'_0[4 r_\phi ]  \| \D_U W_m(U)\|_{\kY\to\kY^s}=1 + \sigma_\Psi[2r_\phi] h
\label{e.DPsimEst}
\end{align}
 uniformly in  $U \in \kB^{2r_\phi}_0$, $m\geq 0$, $h\in[0,h_*]$ for some constant $\sigma_\Psi[2r_\phi]>0$. Plugging these estimates into \eqref{e.helpMain} gives 
 \begin{align}
 E_m^{n+1}(  U^0, h)      \leq \rho h^{p+1} (h m^{p+1-\ell} + m^{p-\ell } )  + (1+\sigma_\Psi h)E_m^n(U^0,h),
      \label{e.mainComp}%  \notag
    \end{align}   
  for some $\rho > 0$, where $\sigma_\Psi = \sigma_\Psi[2r_\phi]$. 
 Hence
  \begin{align*}
    E_m^n&(U,h) \leq \rho h^{p+1}  (h m^{ p+1-\ell } + m^{p-\ell } )\frac{(1+\sigma_\Psi  h)^n}{\sigma_\Psi  h} 
    %\\
  %  &\leq \frac{\rho}{\sigma_\Psi } h^p (h m^{ p+1-\ell } + m^{p-\ell } )\left(1+\frac{n\sigma_\Psi  h}{n}\right)^n 
    %\\& 
    \leq \frac{\rho}{\sigma_\Psi }  h^p (h m^{p+1-\ell } + m^{p-\ell } )e^{n\sigma_\Psi h}.
  \end{align*}
Choosing $m(h) = h^{-1}$ we see that  for $nh\leq T$, $h \in [0,h_*]$  and $\ell\leq p$
\begin{equation}\label{e.glob-ge_m}
  \|(\Psi^h_m)^n(U^0) - \Phi^{nh}_m(U^0) \|_{\kY} \leq
  \frac{\rho}{\sigma_\Psi }\e^{\sigma_\Psi T} h^{p} (2 h^{\ell-p})
  = C\e^{\sigma_\Psi  T}h^{ \ell  }.
\end{equation}
%If $\ell>p$ then $m^{[p-\ell]_+} = O(1)$, so for general $\ell \geq 0$ we obtain
%\begin{equation}\label{e.glob-ge_m}
%  \|(\Psi^h_m)^n(U^0) - \Phi^{nh}_m(U^0) \|_{\kY} \leq   C\e^{\sigma_\Psi  T}h^{ \min(\ell,p)  }.
%\end{equation}
%
Using  \eqref{e.glob-ge_m} we can  ensure  \eqref{e.CondmeanvalueThm}  by possibly reducing $h_*>0$.
 
%%%%%%%%%%%%%%%%%%%%%%%%%%%%%%%%%%%%%%%%%%%%%%%%%%%%%%%%%%%%%%%%%%%%%%%%%%%%

 \smallskip
\noindent
{\em Step 3 (Global truncation error of numerical trajectory)}
We will prove that for $m(h) = h^{-1}$, $nh\leq T$, $h \in [0,h_*]$,
\begin{equation}\label{e.glob-PsiPsihm}
  e^n(U^0)  :=\| (\Psi^h)^n(U^0)-(\Psi_{m(h)}^h)^n(U^0)\|_{\kY} = 
\kO( m^{-\ell } )
\end{equation}
uniformly for initial data $U^0$ satisfying \eqref{e.condPhiGlobal}.
The proof is as in \cite{Wulff2016} for A stable Runge Kutta time discretizations, with the necessary adaptations:
For $n\in \N$, $(n+1)h\leq T$,
\begin{align}
e^{n+1}(U^0) &\leq \| (\Psi^h \circ (\Psi^h)^n)(U^0) -(\Psi^h\circ
(\Psi^h_m)^n)( U^0) \|+ e^1((\Psi^h_m)^n)( U^0)) \notag\\
& \leq    \sup_{U \in \kB^{2r_\psi}_0  } \| \D\Psi^h(U)\|_{ \kY\to \kY}\|  e^n(U^0)\|_{\kY} 
+ \| e^1((\Psi^h_m)^n(U^0))\|_{\kY}
\label{e.en}
\end{align}
provided that 
\begin{align} \label{e.psimPsiTheta}
   \|(\Psi^h_m)^n(U^0)\|_{\kY} \leq r_\psi,   \quad \|e^n(U^0)\|_{\kY} \leq r_\psi, \quad 0\leq nh\leq T.
   \end{align} 
  To obtain the first estimate of \eqref{e.psimPsiTheta}
    we use that for $m=m(h) = h^{- 1}$, 
$nh\leq T$, $h \in [0,h_*]$,   we have
\begin{align}
  \|(\Psi_{m(h)}^h)^n(U^0)\|_{\kY_{\ell}}
&\leq  \|(\Psi^h_m)^n(U^0) - \Phi^{nh}_m(U^0) \|_{\kY_{\ell} }+ \| \Phi^{nh}_m(U^0) \|_{\kY_\ell } \notag\\
&\leq  m^\ell \|(\Psi^h_m)^n(U^0) - \Phi^{nh}_m(U^0) \|_{\kY }+r_\phi
% \notag\\
%& \leq h^{-\ell} C\e^{\sigma_\Psi T }h^{\ell} +r_\phi 
\leq C\e^{\sigma_\Psi T }+r_\phi
= r_\psi
\label{e.psim-Ell}
\end{align}
for some $r_\psi>0$. Here  $r_\phi$ is as in \eqref{e.hatC_ell} and 
we used \eqref{eqn:proj_est}  in the second  and  
\eqref{e.glob-ge_m} in the third inequality.
   
To obtain the second estimate of \eqref{e.psimPsiTheta}
note that  Theorem \ref{thm:num_method_regularity},  with  $R$  replaced by $4r_\psi $,  gives
$
W^i, \Psi \in \kCb^1(\kB_{0}^{2r_\psi}; \kB_{0}^{4r_\psi}).
$
Then \eqref{e.DPsimEst} applies, with $\Psi_m$ replaced by $\Psi$, $W_m$ by $W$ and $r_\phi$ by  $r_\psi$,  
and so \eqref{e.en} for $n\in \N$,  $h\in [0,h_*]$ and  $(n+1)h\leq T$ gives
 \begin{align}
  \|e^{n+1}(U^0) \|_{\kY}
  \leq   (1+\sigma_\Psi h)\|  e^n(U^0)\|_{\kY} + h \kO(m^{-\ell}),
 \label{e.enRecursion}
\end{align} 
where $m=m(h) = h^{-1}$ and $\sigma_\Psi = \sigma_\Psi[2r_\psi]$,
with order constant uniformly in all $U^0$ satisfying \eqref{e.condPhiGlobal},  as long as the second estimate of
\eqref{e.psimPsiTheta} holds.
Here we need that for $U \in \P_m\kY$,
\begin{equation}\label{e.e1}
e^1(U) = h\b^T( hA) \left( \left( \P_m( B(W(U,h))-B(W_m(U,h)) \right) + \Q_m B(W(U,h)) \right),
\end{equation}
so that for $U \in \kB_{\ell}^{r_\psi} \cap \P_m\kY$, $h\in [0,h_*]$, by
\eqref{e.wwm} (with  $R$ replaced by    $2r_\psi$)
\begin{align} \| e^1(U)\|_{\kY }
  & \leq  h M_\b   ( M_0'[2r_\psi]\|W(U,h)-W_m(U,h)\|_{\kY^s}  +\|\Q_m B(W(U,h))\|_{\kY^s})  \notag \\
  & \leq  h M_\b(  M_0'[2r_\psi]\kO(m^{-\ell}) + m^{-\ell} M_\ell[2r_\psi] ) =  h
  \kO(m^{-\ell}),
\label{e.e1Est}
\end{align}
where $m=m(h)$.  

From \eqref{e.enRecursion} we deduce for $nh\leq T$, $h\in [0, h_*]$ and all $U^0$ satisfying \eqref{e.condPhiGlobal} that
\begin{align}
  \|   e^n(U^0)\|_{\kY} 
  %&\leq (1+\sigma_\Psi h)^{n-1}\|  e^1(U^0)\|_{\kY} +    \frac{1}{\sigma_\Psi h}\left((1+\sigma_\Psi h)^{n-1}-1\right)h\kO(m^{-\ell})  \notag\\
  &\leq    \e^{\sigma_\Psi T}(\|  e^1(U^0)\|_{\kY} + \kO(m^{-\ell })) 
  = \kO(m^{-\ell}),\label{e.enEst}
\end{align}
with $m=m(h)$.
Here  we used that    \eqref{e.psipsim} with $R = 2 R_\Phi$  implies
 that $\| e^1(U^0)\|_\kY = O(m^{-\ell})$.
By choosing a possibly  smaller $h_*$  and thereby increasing  $m=h^{-1}$,
the second estimate of \eqref{e.psimPsiTheta} is satisfied. 
This proves \eqref{e.glob-PsiPsihm}.

Hence, \eqref{e.globPhiGalError}, \eqref{e.glob-ge_m} and \eqref{e.glob-PsiPsihm}
show that
%\begin{equation}\label{e.EnEst}
\[
  \| \Phi^{nh}(U^0) -( \Psi^h)^n(U^0)\|_\kY 
  \leq   \| \Phi^t(U^0)-\Phi^t_m(U^0) \|_\kY+  E_m^n(U^0,h) +e^n(U^0) = \kO( h^{\ell} )
  \]
%\end{equation}
for $nh\leq T$, $m(h) = h^{-1}$, $h\in [0, h_*]$ and $U^0$ satisfying \eqref{e.condPhiGlobal}.
\qed

%%%%%%%%%%%%%%%%%%%%%%%%%%%%%%%%%%%%%%%%%%%%%%%%%%%%%%%%%%%%%%%%%%%%%%%%%%%%%%%%%%%%
\section{Error estimates for  exponential Rosenbrock methods}
\label{s.exp_Rosenbrock}
In this section we extend our results to exponential Rosenbrock methods.
As in \cite{SchweitzerHochbruck} we define an   exponential Rosenbrock method
as
\begin{subequations}
\begin{align}
W&= \e^{\c h J(U^0)} \1 U^0 + h  \a(h J(U^0)) G(W,U^0), \label{e.WRosen}\\
U^1 & = \e^{h J(U^0) } U^0 + h   \b^T(h J(U^0)) G(W,U^0),\label{e.U1Rosen}
\end{align}
\end{subequations}
where
\[
J(U^0)= A+ \D B(U^0),\quad  G(U,U^0)  =B(U) - \D B(U^0) U.
\]
Here we define  $(G(W,U^0))^i = G(W^i,U^0)$, $i=1,\ldots, s$, analogously to the definition of $B(W)$, see \eqref{e.defB(W)}.
%\fbox{Do we want explicit only?} and $a_{ij} = 0$ for $i\leq j$.
We need stronger conditions for $\a(z)$ and $\b(z)$ because $J(U^0)$ might  not be a normal operator, so that $\a(hJ(U^0))$ and $\b(hJ(U^0))$ are in general not defined under assumption (EXP).
But we can define $\exp(t J(U^0))$ as flow map for the  evolution equation $\dot X = J(U^0) X$
 on $\kY_\ell$, $\ell \in I$, $\ell\leq N-1$, under assumptions (A2) and (B) by \cite{Pazy}.
We therefore modify condition (EXP) following \cite{SchweitzerHochbruck}:
 \begin{enumerate}
 \item[(EXP')]  For each coefficient $\a_{ij}:  \C \to\C$, $\b_i:\C \to \C$, $i,j=1,\ldots, s$,  there is a sequence $\{\lambda_k\}_{k\in \N}$  with $\lambda_k\geq 0$,  such that these coefficents are  linear combinations of the functions $\varphi_k(\lambda_k z)$, $k\in \N_0$, with $\varphi_k(z) := \int_0^1 \e^{(1-s) z} \frac{ s^{k-1}}{(k-1)!} \d s$, $k\in \N$, and $\phi_0(z) = \e^z$.
 \end{enumerate}
 Since $G$ contains a derivative of $B$  and we need the nonlinearity $G(U,U^0)$ to be $C^1$ on all $\kY_\ell$, $\ell \in I$, we also need to modify condition (B) as follows:
 \begin{enumerate}
 \item[(B')]  (B) holds with $N > \lceil L \rceil+1$.
  \end{enumerate}

 \begin{lemma}[Bound on $\e^{hJ}$ and $\varphi_k(hJ)$] \label{l.ehJ}
 Assume (A1), (A2) and (B') and let $R>0$, $\ell \in I$. Then  for all $U^0\in \kB_\ell^{R}$, $h\geq 0$
 \[
 \| \e^{hJ(U^0)} \|_{\kY_\ell \to \kY_\ell} \leq \e^{M'_\ell[R] h} \quad\mbox{and}\quad 
 \|\varphi_k(hJ(U^0)) \|_{\kY_\ell \to \kY_\ell} \leq \varphi_k(M'_\ell[R] h)~~\mbox{for all}~k\in \N.
 \]  
 \end{lemma}
 \begin{proof}
 We have
    \begin{equation}\label{e.etJ}
  \e^{t J(U^0)}= \e^{t A}  +  \int_0^t \e^{(t-s) A}  \D_U B(U^0)  \exp(s J(U^0))\d s.
  \end{equation}
  Let $x(t) = \| \exp(t J(U^0))\|_{\kY_\ell\to \kY_\ell}$. Then
  \[
  x(t) \leq  1 +  \int_0^t M'_\ell[R] x(s) \d s.
  \]
   Hence $x(t) \leq \e^{t M'_\ell[R]} $. The second estimate follows  from the definition 
   of $\varphi_k$.
  \end{proof}

\begin{lemma}[Bound on $J^j$] \label{l.J^j}
 Assume (A1), (A2) and (B') and let $R>0$, $\ell-1 \in I^-$. 
 Then  for all $j \in \N$, $j\leq \ell$ and all  $U^0\in \kB_\ell^{R}$, $h\geq 0$,
\begin{subequations}
 \begin{align}
 \| J^j(U^0) \|_{\kY_\ell \to \kY_{\ell-j}} &\leq  \prod_{n=1}^{j}  (1+ M'_{\ell-n}[R]),\label{e.J^j}\\
 \|\partial_h^j \varphi_k(hJ(U^0))\|_{\kY_\ell \to \kY_{\ell-j}}   &\leq  \varphi_k^{(j)}(M'_{j-\ell}[R] h) \prod_{n=1}^j  (1+ M'_{\ell-n}[R]).\label{e.partial_h^jphi}
 \end{align}
 \end{subequations}
 \end{lemma}
 \begin{proof}
  We have  
 \begin{align*}
  \| J^j(U^0) \|_{\kY_\ell\to \kY_{\ell-j}} & \leq \prod_{n=1}^j  \|A + \D B(U^0) \|_{\kY_{\ell+1-n} \to \kY_{\ell-n}}
 % \leq \prod_{n=1}^j  ( 1 + \| \D B(U^0) \|_{\kY_{n-1} \to \kY_{n-1}} )\\&
  \leq \prod_{n=1}^j   (1 +M'_{\ell-n}[R]).
 \end{align*}
  The integrand of $\varphi_k^{(j)}(z) = \int_0^1 (1-s)^j \e^{(1-s) z} \frac{ s^{k-1}}{(k-1)!} \d s$
  is   non-negative so that 
  \[
  \|\varphi_k^{(j)}(hJ(U^0)) \|_{\kY_{\ell-j}\to \kY_{\ell-j}} \leq \varphi_k^{(j)}(hM'_{\ell-j}[R]).
  \]
  Hence 
  \begin{align*}
  \|\partial_h^j \varphi_k(hJ(U^0))\|_{\kY_\ell \to \kY_{\ell-j}} &  = \|\varphi_k^{(j)}(t J(U^0))J^j(U^0)\|_{\kY_{\ell} \to \kY_{\ell-j}}\\
  & \leq \varphi_k^{(j)}(hM'_{\ell-j}[R])\| J^j(U^0)\|_{\kY_\ell \to \kY_{\ell-j}}.
  \end{align*}
 \end{proof}
 
 \begin{lemma}[Bounds on $h$ derivatives of $\a(hJ)$ and $\b(hJ)$]\label{l.ahA-rosenbrock}
 Assume (A1), (A2),  (B') and  (EXP') and let $R>0$, $\ell\in I^-$. Then Lemma \ref{l.ahA} on the $h$ derivatives of $\a(hJ(U^0))$ and $\b(hJ(U^0))$ holds true    
 for $U^0 \in \kB^R_\ell$ and all $k\in \N_0$ such that $\ell\geq k$ with bounds $M^{(k)}_{\a,,\ell}[R]$, $M^{(k)}_{\b,\ell}[R]$,  $M^{(k)}_{h,\a,\ell}[R]$, $M^{(k)}_{h,\b,\ell}[R]$ that also depend on $\ell$ and $R$.
 \end{lemma}
 \begin{proof}
 Using (EXP')  this follows from Lemma \ref{l.ehJ} in the case $k=0$ and from \eqref{e.partial_h^jphi} and \eqref{e.Y_ellCompare} in the case $k>0$.
 \end{proof}
 
 We also need bounds on $\D^j_U \e^{h J(U)}$, $\a(h J(U)$ and $\b(h J(U)$:
 \begin{lemma}[Bounds on $U$ derivatives of $\e^{h J(U)}$, $\a(h J(U)$ and $\b(h J(U)$] \label{l.D_UJ}
 Assume (A1), (A2) and (B),  let $R>0$ and $\ell \in I$. Then  for all $j \in \N$, $j+\ell\leq N-1$, and all $U\in \kB_\ell^R$, 
 \begin{equation}\label{e.D_etJ}
  \|  \D^j_U  \e^{t J(U)}\|_{\kY_\ell\to\kY_\ell} \leq P_j(t)  e^{M'_\ell[R] t}
  \end{equation}
  where $P_j(t)$ is a polynomial of degree $j$ in $t$ with $P_j(0)=0$ and coefficients
  which are polynomials in $M^{(n)}_\ell [R]$, $2\leq n\leq j+1$. 
  
  These kinds of bounds also hold for 
  $ \|  \D^j_U  \a(t J(U))\|_{\kY_\ell^s\to\kY_\ell^s}$ and $ \|  \D^j_U  \b({t J(U)})\|_{\kY^s_\ell\to\kY_\ell}$.
 \end{lemma}
% \fbox{Only really needed for $n=1$}
 \begin{proof}
 The last statement follows from \eqref{e.D_etJ} due to  (EXP').  To prove   \eqref{e.D_etJ}, note that for $j\in \N$, $j +\ell \leq N-1$,
  \begin{align*}
  \partial_t \D^j_U  \e^{t J(U)}& 
  %= \D^j_U  \partial_t  \e^{t J(U)}
  = \D^j_U\left(J(U)  \e^{t J(U)}\right)  
%  = \sum_{i=0}^{j}  \D^{j-i}_U J(U) \D^{i}_U   \e^{t J(U)}\\ &
  =  \sum_{i=0}^{j-1}  \D^{j+1-i}_U B(U) \D^{i}_U  \e^{t J(U)} + J(U) \D^j_U   \e^{t J(U)}.
  \end{align*}
 Assume that $\D^i_U   \e^{t J(U)}$ exist for $i\leq j-1$. Then  
  \[
  \D^j_U  \e^{t J(U)} =  \int_0^t \e^{(t-s) J(U)} \sum_{i=0}^{j-1}  \D^{j+1-i}_U B(U) \D^{i}_U   \e^{s J(U)}\d s
  \]
  is well defined on $\kY_\ell$. Moreover,  for $U \in \kB_\ell^R$,
  \[
  \|  \D^j_U  \e^{t J(U)}\|_{\kY_\ell\to\kY_\ell}  \leq   \int_0^t  e^{M'_\ell[R](t-s)}\sum_{i=0}^{j-1}  M^{(j+1-i)}_\ell[R] \|\D^{i}_U  \e^{s J(U)}\|_{\kY_\ell\to\kY_\ell}  \d s
  \]
  For $j=1$ Lemma \ref{l.ehJ} yields
$
  \|  \D_U  \e^{t J(U)}\|_{\kY_\ell\to\kY_\ell}  \leq  
   M^{(2)}_\ell[R] t  e^{M'_\ell[R] t}
 $
 and inductively this gives \eqref{e.D_etJ}:
 Assume that for some $a_{k,n}$ and all $1\leq n\leq j-1$
   \[
  \|  \D^{n}_U  \e^{t J(U)}\|_{\kY_\ell\to\kY_\ell}  \leq e^{M'_\ell[R]t}\sum_{k=1}^n \frac{ a_{k,n} t^k}{k!}.
  \]
  Then 
   \begin{align*}
  \|  \D^j_U  \e^{t J(U)}\|_{\kY_\ell\to\kY_\ell}   &   \leq  e^{M'_\ell[R]t}\int_0^t M^{(j+1)}_\ell[R] \sum_{i=1}^{j-1}  M^{(j+1-i)}_\ell[R]\sum_{k=1}^i\frac{a_{k,i} s^k}{k!} \d s 
 \\
 & =  e^{M't}  (t M^{(j+1)}_\ell[R]   \sum_{i=1}^{j-1}  M^{(j+1-i)}\sum_{k=1}^i\frac{a_{k,i} t^{k+1}}{(k+1)!} )\\
 & =  e^{M'_\ell[R]t}  (t M^{(j+1)}_\ell[R] \sum_{k=1}^j (\sum_{i=k}^{j-1}  M^{(j+1-i)}_\ell[R]\frac{a_{k,i} t^{k+1}}{(k+1)!} ) \\
 & =  e^{M'_\ell[R]t}\sum_{k=1}^{j} \frac{ a_{k,j} t^k}{k!}
  \end{align*}
  with $a_{1,j}= M^{(j+1)}_\ell[R]$ and 
 $
  a_{k,j} = \sum_{i=k-1}^{j-1}  M^{(j+1-i)}_\ell[R] a_{k-1,i} 
 $
  for $k\geq 2$.
   \end{proof}
   
Due to this  lemma Theorem \ref{thm:num_method_regularity} on the regularity of the numerical method holds
true, if $N$ is replaced by $N-1$ in \eqref{eqn:glob-U-num_method_regularity} and (EXP'), (B') is assumed.
 To show that Lemma 
\ref{l.reg-proj-num-method} on the projection error of the numerical method also remains
true under these assumptions we need the following:

\begin{lemma}[Projection error of $\e^{h J(U^0)}$, $\a(h J(U^0))$ and $\b(h J(U^0)) $]\label{l.proj_ahJ} Assume (A1), (A2), (B'), (EXP') and let $R>0$, $h_*>0$, $\ell \in I$. Let  $J_m(U) = \P_m J(\P_m U) \P_m$. Then 
\begin{align*}
\| \e^{h J(U^0)} - \e^{h J_m(U^0)}\|_{\kY_\ell\to \kY} &= \kO(m^{-\ell})\\
\| \a(h J(U^0)) - \a(h J_m(U^0))\|_{\kY_\ell^s\to \kY^s} &= \kO(m^{-\ell}),
\\
\| \b(h J(U^0)) - \b(h J_m(U^0))\|_{\kY_\ell^s\to \kY} &= \kO(m^{-\ell})
\end{align*}
with order constants uniform in  for all  $U^0 \in \kB_{\ell}^{R}$ and $h\in [0,h_*]$.
\end{lemma}
\begin{proof}
Due to (EXP') it is enough to prove the first estimate.  Note that \eqref{e.etJ}
also holds for the Galerkin truncated system
%\begin{align*}
%  \e^{h J(U^0)} - \e^{h J_m(U^0)} =  \Q_m \e^{h J(U^0)} + \P_m ( \e^{h J(U^0)} - \e^{h J_m (U^0)})
 % \end{align*}
%  We have
\begin{align*}
 \e^{tJ_m(U^0)} = \e^{tA_m} + \int_0^t \e^{(t-s)A} \D B_m(U^0) \e^{s J_m(U^0)}\d s,
   \end{align*} 
 where $A= \P_m A$, Hence with $x(t) =  \|\e^{t J(U^0)} -  \e^{tJ_m(U^0)} \|_{\kY_\ell \to \kY}$ 
  \begin{align*}
  x(t) &\leq  \|\e^{tA} -  \e^{tA_m} \|_{\kY_\ell \to \kY}
  + \int_0^t \| (\D B(U^0) -\D B_m(U^0))\e^{s J(U^0)}\|_{\kY_\ell  \to \kY} \d s \\
  &\quad + \int_0^t \| \D B_m(U^0)\|_{\kY\to \kY} x(s)  \d s\\
&  \leq \| \Q_m\|_{\kY_\ell \to \kY} + \int_0^t \| (\D B(U^0) -\D B_m(U^0))\|_{\kY_\ell \to \kY} \e^{s M'_\ell} \d s   + M'_0\int_0^t   x(s)  \d s
  \end{align*}
  where we used Lemma \ref{l.ehJ} and (A2) and $M'_n=M'_n[R]$ for $n\in I$. From \eqref{eqn:proj_est} we know that $\| \Q_m\|_{\kY_\ell \to \kY}  = m^{-\ell}$. Moreover 
  \begin{align*}
  \| (\D B(U^0) -\D B_m(U^0))\|_{\kY_\ell \to \kY}
 & \leq  \|\P_m (\D B(U^0) -\D  B( \P_m U^0))\|_{\kY_\ell \to \kY}\\
 & \quad +  \| \Q_m \D B(U^0)\|_{\kY_\ell \to \kY} \\
 & \leq   M'' \| \Q_m U^0\|_\kY + m^{-\ell}M'_\ell = m^{-\ell}(M'' R+ M'_\ell )
  \end{align*}
  with $M'' = M''_0[R]$.
  Hence
  \begin{align}\label{e.x(t)}
  x(t) \leq  m^{-\ell} + \int_0^t  m^{-\ell}(M'' R+ M'_\ell ) \e^{s M'_\ell} \d s + M'_0\int_0^t   x(s)  \d s
  \end{align}
  and  by Gronwall's lemma
  \[
  \|\e^{t J(U^0)} -  \e^{tJ_m(U^0)} \|_{\kY_\ell \to \kY}= x(t) =\left( m^{-\ell} + m^{-\ell}(M'' R+ M'_\ell )\frac{(\e^{t M'_\ell}-1)}{M'_\ell}\right) \e^{t M'}.
  \]
\end{proof}

\begin{lemma}[Projection error of the exponential Rosenbrock method]
  Assume (A1), (A2), (B') and (EXP').  Then Lemma \ref{l.reg-proj-num-method} remains valid.
 \end{lemma}
 \begin{proof}
 We have, with $J=J(U^0)$, $J_m=J_m(U^0)=\P_m J(\P_m U^0)$,
 \[
  \e^{h\c  J} \1 U - \e^{h\c J_m} \1 \P_m U =
   (\e^{h\c   J} \1 U - \e^{h\c  J } \1 \P_m U) +   
   (\e^{h\c  J } \1 \P_m U- \e^{h\c  J_m} \1 \P_m U)
 \]
 and so by Lemma  \ref{l.ehJ} and Lemma \ref{l.proj_ahJ} for $U\in \kB_\ell^{R/2}$
 \begin{align*} 
 \| \e^{\c h J} \1 U - \e^{h J_m} \1 \P_m U\|_{\kY^s} 
 &\leq \| \e^{\c h J}\|_{\kY^s\to \kY^s} \| \Q_m U\|_{\Y} 
 + \|\e^{h J }  - \e^{h J_m}\|_{\kY_\ell\to \kY} \|  \P_m U\|_{\Y_\ell}\\
 & \leq  \e^{h M'[R/2]} m^{-\ell}\frac{R}2 + \kO(m^{-\ell}) = \kO(m^{-\ell}).
 \end{align*}
 Similarly,   with 
 $G_m(U) = \P_m G(\P_m U)$,
  \begin{align*}
  \a(h J)G -& \a(h J_m) G_m  =  \a(h J)(G-G_m) +( \a(h J)- \a(hJ_m)) G_m\\
  & = \a(h J)(G-\P_m G) + \a(h A)( \P_m G-G_m)+( \a(h J)- \a(hJ_m)) G_m\
   \end{align*}
   so that by Lemma \ref{l.ahA-rosenbrock} and Lemma \ref{l.proj_ahJ}
   using \eqref{eqn:proj_est} and that for $U, \hat U \in \kB_\ell^R$ we have $\|G(U,\hat U)\|_{\kY_\ell} 
   \leq M_\ell + M'_\ell  R$ (with $M_\ell=M_\ell[R]$, $M'_\ell=M'_\ell[R]$) we get for $W=W(U^0)$, $W_m = W_m(U^0)$, $U^0 \in \kB_\ell^{R/2}$,
 \begin{align*} 
\| \a(h J)G(W) - \a(h J_m) G_m(W_m)\|_{\kY^s} & \leq \| \a(h J)\|_{\kY^s\to \kY^s}  (M_\ell + M'_\ell R )m^{-\ell} \\
& \quad + \| \a(h J)\|_{\kY^s\to \kY^s} \| G(W) - G(W_m)\|_{\kY^s}  \\
& \quad+\| \a(h J)- \a(hJ_m)\|_{\kY_\ell^s\to \kY^s} (M_\ell + M'_\ell R )\\
& \leq \kO(m^{-\ell})  + M_{\a,0}[R] \| G(W) - G(W_m)\|_{\kY^s}.
  \end{align*}
  Hence, from $\| \D_U G(U,\hat U)\|_{\kY\to \kY} =  \| \D_U B(U)- \D_U B(\hat U)\|_{\kY\to \kY}\leq 2M'[R]$ for $U,\hat U\in \kB_0^R$  we get, using \eqref{e.WRosen} for both $W$ and $W_m$, similarly as in \eqref{e.WWm-est},
  \[
  \| W-W_m\|_{\kY^s} \leq \kO(m^{-\ell})  + 2h M_{\a,0}[R] M'_0[R]\| W-W_m\|_{\kY^s}
  \]
  which for $h\in [0,h_*]$ with $h_*$ small  enough shows \eqref{e.wwm} for exponential Rosenbrock methods. 
 For $\b(hJ)$ we obtain in the same way that 
 \begin{align*} 
\| \b(h J)G - \b(h J_m) G_m\|_{\kY} &  \leq 
 \kO(m^{-\ell})  + M_{\b,0}[R] \| G(W) - G(W_m)\|_{\kY^s} \\
 & \leq  \kO(m^{-\ell})  + 2M_{\b,0}[R] M'_0[R] \| W-W_m\|_{\kY^s}.
 \end{align*}
 Using \eqref{e.U1Rosen} for both $\Psi$ and $\Psi_m$  and    \eqref{e.wwm} for exponential Rosenbrock methods gives \eqref{e.psipsim} for exponential Rosenbrock methods as well.
 \end{proof}
 
 \begin{lemma}[$m$ dependent bounds on $W_m$   for exponential Rosenbrock methods]
 Lemma \ref{lem:wm-m-dep-deriv} remains true for exponential Rosenbrock methods if (A1), (A2), (B') and (EXP') hold.
 \end{lemma}
 \begin{proof}
 We have to replace $A$ by $J_m=J_m(U)$ and $B(W_m)$ by $G=G(W_m, U)$ throughout the proof of  Lemma \ref{lem:wm-m-dep-deriv}.
 \eqref{e.firstterm} becomes for $\ell \in I^-$, $k\geq \ell$
 \begin{align*}
  \sup_{h\in [0,h_*]}&
  \| \partial _h^k(\exp(\c h J_m) \P_m \|_{\kY_\ell^s\to \kY^s}  =
  \sup_{h\in [0,h_*]} \|  \c^k J_m^k \exp(\c h  J_m) \P_m \|_{\kY_\ell^s\to \kY^s} \notag\\
  &\leq  \|  \c^k\|  \sup_{h\in [0,h_*]} \|\exp(\c h  J_m)\|_{\kY^s\to \kY^s} \|   J_m^k   \|_{\kY^s_\ell \to \kY^s} 
  \\
  & \leq  \|  \c^k\| \e^{ \|\c\| M' h_*}\|   J_m^{\lfloor \ell\rfloor}   \|_{\kY_\ell \to \kY_{\ell-\lfloor \ell\rfloor} }\|   J_m^{k-\lfloor \ell\rfloor} \|_{\kY_{\ell-\lfloor \ell\rfloor}\to \kY}=
   O( m^{k-\ell}).
   \end{align*}
    Here we used \eqref{e.J^j} and Lemma  \ref{l.ehJ} with $J$ replaced by $J_m$ noting that the proof holds true with the same bounds. Moreover we used that for $k\geq  \ell$
   \begin{align*}
\|   J_m^{k-\lfloor \ell\rfloor} \|_{\kY_{\ell-\lfloor \ell\rfloor}\to \kY} & 
 \leq   \|   J_m^{k-\lfloor \ell\rfloor-1} \|_{\kY\to \kY}
 \|   J_m \|_{\kY_{\ell-\lfloor \ell\rfloor}\to \kY}
 = O( m^{k-\lfloor \ell\rfloor-1})O(m^{1-(\ell-\lfloor \ell\rfloor)}) 
 \\
 & =  O( m^{k-\ell}).
   \end{align*}
  We replace the bound $M_{h,\a}^{(k-j)}$ by $M_{h,\a,k-j}^{(k-j)}[R]$ in \eqref{e.k-jDerivha}, \eqref{e.jLessEll-W}  and by $M_{h,\a,0}^{(k-j)}[R]$ in \eqref{e.jGreaterEll-W} and  replace the bound $M_\a$      by $M_{\a,0}[R]$ in \eqref{e.derivwmest}.
 Then \eqref{e.k-jDerivha} holds true with $A$ replaced by $J_m$  by Lemma \ref{l.ahA-rosenbrock}
 which also applies to $J_m$ with the same bounds.  In the analogue of \eqref{e.Bwmderivest} we have to replace $M_0'[R]$ by $2M_0'[R]$ since this is the required bound for $\D G(W_m(U,h), U)$.
 \end{proof}

  \begin{lemma}[$m$ dependent bounds on $\Psi_m$   for exponential Rosenbrock methods]
 Lemma \ref{lem:psim-m-dep-deriv} remains true for exponential Rosenbrock methods if (A1), (A2), (B') and (EXP') hold.
 \end{lemma}
 \begin{proof}
 The proof of Lemma \ref{lem:psim-m-dep-deriv} remains valid if we replace $A$ by $J_m(U)$ and  $B(W_m(U,h))$ by $G(W_m(U,h), U) $ throughout the proof  and in \eqref{e.k-jDerivhb} replace $M_{\b}^{(i)}$ by $M_{\b,i}^{(i)}[R]$, $i=n-1,n$.
  \end{proof}
 
 \begin{theorem}[Trajectory error for nonsmooth data  of exponential Rosenbrock methods]  Theorem \ref{thm:expInt_convergence}     remains  true  for exponential Rosenbrock methods under conditions (A1), (A2), (B') and (EXP').
 \end{theorem}
 \begin{proof}
 Most of the proof of  Theorem \ref{thm:expInt_convergence} carries over
In the analogue of \eqref{e.DPsimEst}  we   use Lemma \ref{l.D_UJ}  to obtain for $U \in \kB_0^{2r_\phi}$, 
\begin{align*}
 \| \D_U\Psi_m^h(U)\|_{\kY\to\kY}  & \leq   \|\exp(hJ_m)\|_{ \kY\to \kY} + \| \D_U\e^{h J_m}\|_{ \kY\to \kY} \|U\|_\kY \\
 & \quad + h \|\D_U \b(hJ_m )\|_{\kY^s\to \kY}2M_0 + h  M_{\b,0}  M'_0   \| \D_U W_m(U)\|_{ \kY\to\kY^s}
\\
&  =: 1 + \sigma_\Psi[2r_\phi] h
\end{align*}
where constants are evaluated with radius $R=4 r_\phi $.
In \eqref{e.e1} and \eqref{e.e1Est} we have to replace $B(W(U,h))$ by $G(W(U,h))$, and in \eqref{e.e1Est} also $M_\b$  by $M_{\b,0}$ and $M'_0$ by $2M'_0$.  
\end{proof}

%%%%%%%%%%%%%%%%%%%%%%%%%%%%%%%%%%%%%%%%%%%%%%%%%%%%%%%%%%%%%%%%%%%%%%%%%%%%

%\begin{remark}
%\fbox{maybe to do:} exponential Euler-Rosenbrock order 2 result also applies.
%\end{remark}

\section*{Acknowledgement}
I would like to thank Andrew Sadler for preparing the code to test the accuracy 
of the exponential Euler method for non-smooth initial data.

%%%%%%%%%%%%%%%%%%%%%%%%%%%%%%%%%%%%%%%%%%%%%%%%%%%%%%%%%%%%%%%%%%%%%%%%%%%%%%%%%%%%%%%%%%%%%%%%%%%

\end{document}